\documentclass{amsproc}
\usepackage{graphicx,mathrsfs,amssymb,amsmath,color}

\newtheorem{theorem}{Theorem}[section]
\newtheorem{lemma}[theorem]{Lemma}
\newtheorem{corollary}[theorem]{Corollary}
\newtheorem{proposition}[theorem]{Proposition}
\newtheorem{remark}[theorem]{Remark}
\newtheorem{definition}[theorem]{Definition}
\newtheorem{example}[theorem]{Example}

\numberwithin{equation}{section}

\newcommand{\cz}{{\mathbb C}}

\newcommand{\nz}{{\mathbb N}}
\newcommand{\rz}{{\mathbb R}}
\newcommand{\sz}{{\mathbb S}}

\newcommand{\bfa}{\mathbf{a}}

\newcommand{\bfc}{\mathbf{c}}
\newcommand{\bfg}{\mathbf{g}}
\newcommand{\bfh}{\mathbf{h}}

\newcommand{\bfr}{\mathbf{r}}

\newcommand{\bfs}{\mathbf{s}}

\newcommand{\calA}{\mathcal{A}}
\newcommand{\calB}{\mathcal{B}}
\newcommand{\calE}{\mathcal{E}}
\newcommand{\calK}{\mathcal{K}}

\newcommand{\calW}{\mathcal{W}}

\newcommand{\scrC}{\mathscr{C}}
\newcommand{\scrD}{\mathscr{D}}
\newcommand{\scrF}{\mathscr{F}}
\newcommand{\scrH}{\mathscr{H}}
\newcommand{\scrK}{\mathscr{K}}
\newcommand{\scrL}{\mathscr{L}}
\newcommand{\scrS}{\mathscr{S}}

\newcommand{\bfkappa}{\boldsymbol{\kappa}}
\newcommand{\bfmu}{\boldsymbol{\mu}}

\newcommand{\cl}{\mathrm{cl}}
\newcommand{\dbar}{d\hspace*{-0.08em}\bar{}\hspace*{0.1em}}
\newcommand{\eps}{\varepsilon}
\newcommand{\forget}[1]{}

\newcommand{\lra}{\longrightarrow}

\newcommand{\op}{\mathrm{op}}

\newcommand{\rpbar}{\overline{\rz}_+}

\newcommand{\spk}[1]{\langle#1\rangle}

\newcommand{\wh}{\widehat}

\newcommand{\wt}{\widetilde}

\begin{document}
\title[Green Operators in the Edge Algebra Formalism]{
Singular Green Operators of Finite Regularity\\ in the Edge Algebra Formalism}

\author{J\"org Seiler}
\address{Dipartimento di Matematica, Universit\`{a} di Torino, Italy}
\email{joerg.seiler@unito.it}

\begin{abstract}
We introduce a calculus for parameter-dependent singular Green operators on the half-space $\rz^n_+$ that combines both elements of Grubb's calculus for boundary value problems of finite regularity 
and techniques of Schulze's calculus for pseudodifferential operators on manifolds with edges. 
\end{abstract}

\maketitle

\tableofcontents

\section{Introduction}\label{sec:intro}

Since the seminal papers of Seeley \cite{Seel67}, \cite{Seel69} on complex powers of elliptic operators and resolvents of boundary value problems, respectively, it became a broadly applied technique in partial differential equations and geometric ana\-lysis to exploit the pseudodifferential structure of resolvents or other parameter-dependent families of operators. Applications are quite manifold, including resolvent estimates and resolvent trace asymptotics, complex and imaginary powers, $H_\infty$-calculus, maximal $L^p$-regularity, heat kernels and heat trace asymptotics, index theory, spectral asymptotics, $\zeta$-functions and regularized determinants. 

In \cite{Bout1}, Boutet de Monvel introduced a calculus for operators on manifolds with boundary. It contains all classical differential boundary value problems and, in case of Shapiro-Lopatinskii ellipticity, Fredholm inverses can be constructed in the calculus. If we locally model the manifold as a half-plane $\rz^n_+$ with boundary $\rz^{n-1}$, one considers operators of the form 
\begin{align}\label{eq:block-matrix}
 \begin{pmatrix}A_++G & K\\ T & Q\end{pmatrix}: 
   \begin{matrix}\scrS(\rz^n_+,\cz^L)\\ \oplus\\ \scrS(\rz^{n-1},\cz^M)\end{matrix}
   \lra 
   \begin{matrix}\scrS(\rz^n_+,\cz^{L^\prime})\\ \oplus\\ 
   \scrS(\rz^{n-1},\cz^{M^\prime})\end{matrix},
\end{align}
acting between spaces of rapidly decreasing functions $($extendable to Sobolev and Besov spaces$)$; here $L,M,L^\prime,M^\prime$ are non-negative integers, modelling involved vector bundles over the manifold $($these numbers are allowed to be zero$)$. Moreover, $A_+$ is the restriction of a pseudodifferential operator to the half-space, $K$ is a so-called potential or Poisson operator, $T$ is a trace operator, and $Q$ is a pseudodifferential operator on the boundary. In the terminology of \cite{Bout1}, such matrices are called Green operators, while those with $A_+=0$ are termed singular Green operators. 
In the light of what has been said above, it is natural to seek a corresponding calculus for  parameter-dependent boundary value problems, where parameter-elliptic elements possess an inverse (for large parameter) in the calculus. Typically, the parameter ranges in a conical subset of the Euclidean space; in what follows we focus for simplicity on a real parameter $\mu\ge0$.  

Such a calculus has been realized by Grubb in the middle 1980's. A complete exposition one can find in \cite{Grub}, where also applications as mentioned above are described. In particular, the calculus allows the construction of resolvents of pseudodifferential (and non-differential) operators, subject to boundary conditions. One of the key features in this calculus is to start out from operators $A(\mu)=a(x,D,\mu)$ build upon symbols of finite regularity, i.e., satisfying uniform estimates of the form 
\begin{equation*}
|D^\beta_x D^\alpha_\xi D^j_\mu a(x,\xi,\mu)| \le C_{\alpha\beta j}
\big[\spk{\xi}^{\nu-|\alpha|}+\spk{\xi,\mu}^{\nu-|\alpha|}\big]\spk{\xi,\mu}^{d-\nu-j} 
\end{equation*}
for arbitrary orders of derivatives. Here, $d$ is the order of the symbol, and $\nu$ its so-called regularity (note that regularity in this context does \textit{not} refer to smoothness of the symbols in the $x$-variable). This structure induces corresponding notions of order and regularity in the class of singular Green operators. 

In the early 1980's Schulze initiated his programme on calculi of pseudodifferential operators on manifolds with singularities, cf. \cite{ReSc}, \cite{Schu4}, \cite{Schu91}. In his approach, manifolds with boundary can be interpreted as a special case of a manifold with edge. Near the edge, a manifold with edge has the structure of a fibre-bundle over a smooth manifold with fibre being a cone over some base manifold; if the base is a single point one obtains a manifold with boundary. The operators considered on a manifold with edge are  more general than those on a manifold with boundary, in particular, there is no concept of transmission condition involved. The resulting pseudodifferential calculus, the so-called edge algebra, again consists of block-matrices similar to \eqref{eq:block-matrix}, but with $A_+$ replaced by more general operators, and a different class of singular Green operators. We cannot go into details here, but only  want to point out one key ingredient, which is significant for the present paper, namely the systematic use of operator-valued pseudodifferential symbols acting in Hilbert spaces equipped with a strongly continuous group action. As it turns out, singular Green operators have a very elegant and transparent description in this set-up. In the end, the calculus for such operator-valued symbols is nearly identical with that of usual scalar-valued pseudodifferential symbols in Euclidean space (as laid out, for example, in the text-book of Kumano-go \cite{Kuma}). 
Schulze also introduced a parameter-dependent version of the edge algebra. It allows the construction of resolvents of differential operators, subject to edge conditions.  However, pseudodifferential (non-differential) operators are not covered in this setting. A parameter-dependent version of Boutet de Monvel's calculus using the formalism of the edge algebra is described in Schrohe, Schulze \cite{ScSc1}. 

In the present paper we introduce a calculus of singular Green operators on the half-space $\rz^n_+$, combining both Grubb's concept of operators of finite regularity and Schulze's concept of operator-valued pseudodifferential symbols. Grubb's class of singular Green operators contains two natural sub-classes of operators: those of infinite regularity $($which we refer to as strongly parameter-dependent operators$)$ and so-called weakly parameter-dependent operators. One of our key observations is that weakly parameter-dependent operators can be characterized in the framework of operator-valued symbols in the spirit of Schulze's theory $($this is also true for strongly parameter-dependent operators, but this was already known, cf. \cite{Schr01} for example). In Section \ref{sec:poisson} we show this result in detail for the class of Poisson operators. We then implement the rsulting structure of operator-valued symbols in our class of singular Green operators which, by definition, consists in the sum of the weakly and strongly parameter-dependent operators. Doing so, it does not coincide with Grubb's class but is a subclass of it. We show that our class admits a full calculus and do recover many results stated in \cite{Grub} in an alternative way. In particular, this concerns the parametrix construction discussed in Section \ref{sec:param}.

We hope that our findings could prove useful in a possible generalization of the concept of finite regularity from boundary value problems to the edge algebra, thus enabling the treatment of resolvents of pseudodifferential operators on manifolds with edges. We plan to address this in future research. 

The paper is structured as follows: In Section \ref{sec:opval} we introduce some classes of operator-valued symbols in Hilbert spaces with group actions, called strongly and weakly parameter-dependent symbols, and discuss their calculus. The sum of such symbols leads to a new class whose calculus is discussed in Section \ref{sec:finreg}. After the above mentioned analysis of Poisson operators in Section \ref{sec:poisson} we introduce our class of singular Green operators in Section \ref{sec:green}. Section \ref{sec:param} is devoted to the paramtrix construction of operators of the form ``identity plus singular Green operator'' of positive regularity. Section \ref{sec:green03} introduces singular Green operators of positive type and Section \ref{sec:Sobolev} concerns the action of singular Green operators in Sobolev spaces. In the first section of the appendix, Section \ref{sec:app02}, we prove a characterization of pseudodifferential operators with operator-valued symbols by iterated commutators in the sense of the classical result of Beals \cite{Beal77}, \cite{Beal79}, which we need in the parametrix construction. In the final Section \ref{sec:app01} we recall the definition of certain scales of Sobolev spaces on the half-axis which will be used frequently throughout the text. 

\textbf{Important notation:} For $x\in\rz^n$ and $y\in\rz^m$ we shall write 
 $$\spk{x}=(1+|x|^2)^{1/2},\qquad \spk{x,y}=\spk{(x,y)}=(1+|x|^2+|y|^2)^{1/2}.$$ 
For two functions $f,g$ defined on some set $\Omega$ the notation $f\preceq g$ or $f(y)\preceq g(y)$ means that there exists a constant $C\ge 0$ such that $f(y)\le Cg(y)$ for all $y\in\Omega$. In a chain like $f\preceq g\preceq h$, the constant may be different in any inequality
\section{Operator-valued symbols and their calculus}\label{sec:opval}

Let $E$ be a Hilbert space. A group-action is a map $\kappa:(0,+\infty)\to\scrL(E)$ which is 
continuous in the strong operator topology and satisfies $\kappa_1=1$ as well as  
$\kappa_\lambda\kappa_\sigma=\kappa_{\lambda\sigma}$ for every $\lambda,\sigma>0$. 
A standard result from semigroup theory ensures the existence of a constant $M=M(\kappa)\ge0$ 
such that 
\begin{equation}\label{eq:M-kappa}
 \|\kappa_\lambda\|_{\scrL(E)}\preceq \max(\lambda,\lambda^{-1})^M.
\end{equation} 
Given such a group-action, we shall write 
 $$\kappa(y)=\kappa_{\spk{y}},\quad \kappa^{-1}(y)=\kappa^{-1}_{\spk{y}}=\kappa_{1/\spk{y}}\qquad  (y\in\rz^m).$$  

For the present paper, the following example of a group-action is fundamental.  

\begin{example}\label{ex:standard-group-action}
For smooth scalar-valued functions $\phi$ defined on $\rz_+$ $[$or $\rz]$ with compact support we define 
 $$(\kappa_\lambda \phi)(t)=\lambda^{1/2}\phi(\lambda t),\qquad \lambda>0.$$
We extend $\kappa_\lambda$ to act on the space of distributions $\scrD^\prime(\rz_+)$ $[$or $\scrD^\prime(\rz)]$ by 
 $$\spk{\kappa_\lambda T,\phi}=\spk{T,\kappa_\lambda^{-1}\phi}.$$
We shall call this the ''standard group-action'' on $\rz_+$ $[$or $\rz]$; its restriction to various function spaces defines a group-action in the above sense, for example, the restriction to the Sobolev spaces $H^s(\rz_+)$ $[$or $H^s(\rz)]$. Note that it defines a group of unitary operators 
on the standard $L^2$-spaces. 
\end{example}

In the following let $E_0$, $E_1$, $E_2$ be Hilbert spaces equipped with group actions 
$\kappa_0$, $\kappa_1$, and $\kappa_2$, respectively. 

\subsection{Strongly parameter-dependent symbols}

We recall the definition of parameter-dependent symbols in the sense of Schulze. 

\begin{definition}
Let $d\in\rz$. Then $S^d_{1,0}(\rz^{n-1}\times\rpbar;E_0,E_1)$ 
is the space of all smooth functions $a(x,\xi;\mu):\rz^n\times\rz^n\to\scrL(E_0,E_1)$ satisfying 
\begin{equation}\label{eq:symb01}
   \|\kappa^{-1}_1(\xi,\mu)\big\{D^\alpha_{\xi}D^\beta_{x}D^j_\mu
   a(x,\xi;\mu)\big\}\kappa_0(\xi,\mu)\|_{\scrL(E_0,E_1)}
   \preceq \spk{\xi,\mu}^{d-|\alpha|-j}.
\end{equation} 
for every order of derivatives. 
\end{definition} 

Due to \eqref{eq:M-kappa}, the associated space of regularizing symbols 
$S^{-\infty}(\rz^{n}\times\rpbar;E_0,E_1)$, defined by taking the intersection over all $d\in\rz$, consists of those symbols $a$ satisfying 
 $$\|D^\alpha_{\xi}D^\beta_{x}D^j_\mu a(x,\xi;\mu)\|_{\scrL(E_0,E_1)}\preceq \spk{\xi,\mu}^{-N}$$ 
for every $N\ge0$ and every order of derivatives. Obviously, it does not depend on the involved semi-groups. In other terms, 
 $$S^{-\infty}(\rz^{n}\times\rpbar;E_0,E_1)
   =\scrC^\infty_b(\rz^{n}_{x},\scrL(E_0,E_1))\,\wh{\otimes}_\pi\,\scrS(\rz^{n}\times\rpbar),$$ 
the completed projective tensor-product of the space of smooth $\scrL(E_0,E_1)$-valued functions with bounded 
derivatives of arbitrary order and the space of rapidly decreasing functions. 

\subsubsection{Homogeneous functions}

Let $V\subseteq (\rz^{n}\times\rpbar)\setminus\{0\}$ be a conical set. 
A function $a:\rz^{n}\times V\to\scrL(E_0,E_1)$ is called $\kappa$-homogeneous $($or also twisted 
homogeneous$)$ of degree $d\in\rz$ if 
$$a(x,\lambda\xi;\lambda\mu)=\lambda^d \kappa_{1,\lambda} a(x,\xi;\mu)\kappa_{0,\lambda}^{-1} 
\qquad\forall\;(\xi,\mu)\in V\quad\forall\,\lambda>0.$$

\begin{definition}
Let us denote by $S^{(d)}(\rz^{n}\times\rpbar;E_0,E_1)$ the space of all smooth functions 
$a:\rz^{n}\times V\to\scrL(E_0,E_1)$ with $V=(\rz^{n}\times\rpbar)\setminus\{0\}$ 
that are $\kappa$-homogeneous of degree $d$ and satisfy, for every order of derivatives, 
\begin{equation}\label{eq:symb01_hom}
\|\kappa^{-1}_{1,|\xi,\mu|}\big\{D^\alpha_{\xi}D^\beta_{x}D^j_\mu
a(x,\xi;\mu)\big\}\kappa_{0,|\xi,\mu|}\|_{\scrL(E_0,E_1)}
\preceq |\xi,\mu|^{d-|\alpha|-j}
\end{equation} 
\end{definition}

In particular, if $a\in S^{(d)}(\rz^{n}\times\rpbar;E_0,E_1)$ then 
 $$a(x,\xi;\mu)=|\xi,\mu|^d 
   \kappa_{1,|\xi,\mu|}
   a\Big(x,\frac{\xi}{|\xi,\mu|};\frac{\mu}{|\xi,\mu|}\Big)
   \kappa_{0,|\xi,\mu|}^{-1}.$$

If $\chi(\xi,\mu)$ is an arbitrary $0$-excision function and if 
$a\in S^{(d)}(\rz^{n}\times\rpbar;E_0,E_1)$ then 
$\chi a\in S^{d}_{1,0}(\rz^{n}\times\rpbar;E_0,E_1)$. 

\subsubsection{Classical symbols}

Classical symbols (also called poly-homogeneous symbols) have an asymptotic expansion in homogeneous components. 

\begin{definition}
$S^d(\rz^{n}\times\rpbar;E_0,E_1)$ denotes the space of all symbols 
$a\in S^d_{1,0}(\rz^{n}\times\rpbar;E_0,E_1)$ for which exists a sequence 
$a^{(d-j)}\in S^{(d-j)}(\rz^{n}\times\rpbar;E_0,E_1)$, $j\in\nz_0$, such that, for every $N\in\nz$,  
 $$a-\sum_{j=0}^{N-1}\chi(\xi,\mu) a^{(d-j)}\;\in\; S^{d-N}_{1,0}(\rz^{n}\times\rpbar;E_0,E_1).$$ 
\end{definition}

If $a\in S^d(\rz^{n}\times\rpbar;E_0,E_1)$ is as in the previous definition, then $a^{(d)}$ is 
called the homogeneous principal symbol of $a$. It is well defined, since 
 $$a^{(d)}(x,\xi;\mu)=\lim_{\lambda\to+\infty}
   \lambda^{-d}\kappa_{1,\lambda}^{-1}a(x,\lambda\xi;\lambda\mu)\kappa_{0,\lambda},
   \qquad (\xi,\mu)\not=0.$$

\subsection{Weakly parameter-dependent symbols}

Roughly speaking, strong para\-meter-dependence means considering the parameter $\mu$ together with the co-variable $\xi$ as a new joint co-variable $\eta:=(\xi,\mu)$ and then to impose the standard estimates  of pseudodifferential symbols in $\eta$. We now introduce a class where the behaviour is different$:$ differentiation with respect to $\xi$ only improves the decay in $\xi$; still, differentiation with respect to $\mu$ improves the decay in $\eta=(\xi,\mu)$. We shall use the terminology ''weak'' parameter-dependence. 

\begin{definition}\label{def:weakly}
Let $d,\nu\in\rz$. Then $\wt{S}^{d,\nu}_{1,0}(\rz^{n}\times\rpbar;E_0,E_1)_{\bfkappa,\bfkappa}$ 
denotes the space of all symbols satisfying the estimates 
 $$\|\kappa^{-1}_1(\xi,\mu)\big\{D^\alpha_{\xi}D^\beta_{x}D^j_\mu
   a(x,\xi;\mu)\big\}\kappa_0(\xi,\mu)\|_{\scrL(E_0,E_1)}
   \preceq \spk{\xi}^{\nu-|\alpha|}\spk{\xi,\mu}^{d-\nu-j}.$$
\end{definition} 

The reason for using the subscript ``$\bfkappa,\bfkappa$" will become clear later on, 
cf. Section \ref{sec:green01}. 
The corresponding class of regularizing symbols is defined by 
 $$\wt{S}^{d-\infty,\nu-\infty}(\rz^{n}\times\rpbar;E_0,E_1)_{\bfkappa,\bfkappa}
   :=\mathop{\mbox{\Large$\cap$}}_{N\in\rz}\wt{S}^{d-N,\nu-N}_{1,0}
    (\rz^{n}\times\rpbar;E_0,E_1)_{\bfkappa,\bfkappa}.$$
It consists of those symbols with 
 $$\|\kappa^{-1}_1(\xi,\mu)\big\{D^\alpha_{\xi}D^\beta_{x}D^j_\mu
a(x,\xi;\mu)\big\}\kappa_0(\xi,\mu)\|_{\scrL(E_0,E_1)}
\preceq \spk{\xi}^{-N}\spk{\xi,\mu}^{d-\nu-j}$$
for arbitrary $N\ge 0$. Due to \eqref{eq:M-kappa}, since 
$1\preceq \spk{\mu}\spk{\xi,\mu}^{-1}\preceq\spk{\xi}$ and since $N$ is arbitrary, we may equivalently use the estimates 
 $$\|\kappa^{-1}_1(\mu)\big\{D^\alpha_{\xi}D^\beta_{x}D^j_\mu
a(x,\xi;\mu)\big\}\kappa_0(\mu)\|_{\scrL(E_0,E_1)}
\preceq \spk{\xi}^{-N}\spk{\mu}^{d-\nu-j}$$
for arbitrary $N\ge 0$. Note that this space generally depends on the involved group-actions. 

\begin{remark}
Asymptotic summation of a sequence of symbols $a_k$ of order $d-k$ and regularity $\nu-k$ is well-defined. In fact, 
if $\chi(\xi)$ is a $0$-excision function and $0<\eps_k\xrightarrow{k\to+\infty}0$ sufficiently fast, then the series 
 $$a(x,\xi;\mu)=\sum_{k=0}^{+\infty} \chi(\eps_k\xi) a_k(x,\xi;\mu)$$
converges in $\wt{S}^{d,\nu}_{1,0}(\rz^{n}\times\rpbar;E_0,E_1)_{\bfkappa,\bfkappa}$  and 
$a-\sum_{k=0}^{N-1} a_k$ belongs to $\wt{S}^{d-N,\nu-N}_{1,0}(\rz^{n}\times\rpbar;E_0,E_1)_{\bfkappa,\bfkappa}$ 
for every $N$. 
\end{remark}

\subsubsection{Classical symbols} 

Also in case of weak parameter-dependence we can introduce the subclass of classical symbols. 

\begin{definition}\label{def:kappa-hom}
Let $\wt{S}^{(d,\nu)}(\rz^{n-1}\times\rpbar;E_0,E_1)_{\bfkappa,\bfkappa}$ 
be the space of all smooth functions 
$\wt{a}:\rz^{n}\times V\to\scrL(E_0,E_1)$ with $V=(\rz^{n}\setminus\{0\})\times\rpbar$ that are
$\kappa$-homogeneous of degree $d$, and satisfy the estimates   
\begin{equation}\label{eq:symb02_hom}
\|\kappa^{-1}_{1,|\xi,\mu|}\big\{D^\alpha_{\xi}D^\beta_{x}D^j_\mu
\wt{a}(x,\xi;\mu)\big\}\kappa_{0,|\xi,\mu|}\|_{\scrL(E_0,E_1)}
\preceq |\xi|^{\nu-|\alpha|} |\xi,\mu|^{d-\nu-j}.
\end{equation} 
\end{definition}

Note that homogeneous functions in the present context are only defined for $\xi\not=0$ while, 
in case of strong parameter-dependence, they were defined whenever $(\xi,\mu)\not=0$. 
If $\chi(\xi)$ is an arbitrary $0$-excision function and if 
$\wt{a}\in \wt{S}^{(d,\nu)}(\rz^{n}\times\rpbar;E_0,E_1)_{\bfkappa,\bfkappa}$ 
then it is easy to verify that 
$\chi\wt{a}\in \wt{S}^{d,\nu}_{1,0}(\rz^{n}\times\rpbar;E_0,E_1)_{\bfkappa,\bfkappa}$.

\begin{definition}\label{def:kappa-cl}
	$\wt{S}^{d,\nu}(\rz^{n}\times\rpbar;E_0,E_1)_{\bfkappa,\bfkappa}$ denotes the space of all 
	$\wt{a}\in \wt{S}^{d,\nu}_{1,0}(\rz^{n}\times\rpbar;E_0,E_1)_{\bfkappa,\bfkappa}$ which admit a sequence 
	$\wt{a}^{(d-j,\nu-j)}\in S^{(d-j,\nu-j)}(\rz^{n}\times\rpbar;E_0,E_1)_{\bfkappa,\bfkappa}$, $j\in\nz_0$, such that 
	$$\wt{a}-\sum_{j=0}^{N-1}\chi(\xi) \wt{a}^{(d-j,\nu-j)}\;\in\; 
	   \wt{S}^{d-N,\nu-N}_{1,0}(\rz^{n}\times\rpbar;E_0,E_1)_{\bfkappa,\bfkappa}$$
	for every $N\in\nz_0$. 
\end{definition}

If $\wt{a}\in \wt{S}^{d,\nu}(\rz^{n}\times\rpbar;E_0,E_1)_{\bfkappa,\bfkappa}$ 
is as in the previous definition, then $\wt{a}^{(\nu,d)}$ is 
called the homogeneous principal symbol of $a$. It is well defined, since 
$$\wt{a}^{(d,\nu)}(x,\xi;\mu)=\lim_{\lambda\to+\infty}
  \lambda^{-d}\kappa_{1,\lambda}^{-1}\wt{a}(x,\lambda\xi;\lambda\mu)\kappa_{0,\lambda},
  \qquad \xi\not=0.$$
To this end note that if $\wt{r}$ belongs to $\wt{S}^{d-1,\nu-1}_{1,0}(\rz^{n}\times\rpbar;E_0,E_1)_{\bfkappa,\bfkappa}$ then 
we can estimate 
\begin{align*}
 \|\kappa_{1,\lambda}^{-1}\wt{r}(x,\lambda\xi;\lambda\mu) \kappa_{0,\lambda}\|
 &=
 \|\kappa_{1,\spk{\lambda\xi,\lambda\mu}/\lambda}
 \kappa_1^{-1}(\lambda\xi,\lambda\mu)\wt{r}(x,\lambda\xi;\lambda\mu) 
 \kappa_0(\lambda\xi,\lambda\mu) \kappa_{0,\spk{\lambda\xi,\lambda\mu}/\lambda}^{-1}\|\\
 &\preceq 
 \max\Big(\frac{\spk{\lambda\xi,\lambda\mu}}{\lambda},
 \frac{\lambda}{\spk{\lambda\xi,\lambda\mu}}\Big)^M
 \Big(\frac{\spk{\lambda\xi}}{\spk{\lambda\xi,\lambda\mu}}\Big)^\nu
 \spk{\lambda\xi}^{-1}\spk{\lambda\xi,\lambda\mu}^d
\end{align*}
with some $M\ge0$, cf. \eqref{eq:M-kappa}; 
for fixed $(\xi,\mu)$ with $\xi\not=0$, the right-hand side behaves like $\lambda^{d-1}$ for 
$\lambda\to+\infty$. 

\begin{remark}\label{rem:extension}
Let $\wt{a}\in\wt{S}^{(d,\nu)}(\rz^{n}\times\rpbar;E_0,E_1)_{\bfkappa,\bfkappa}$. In a small neighborhood of a point $(0,\mu_0)$ with $\mu_0>0$ we have that 
$\|\wt{a}(x,\xi,\mu)\|_{\scrL(E_0,E_1)}\preceq |\xi|^{\nu}$. Thus if $\nu>0$, 
it follows that 
 $$\lim_{\xi\to 0}\wt{a}(x,\xi;\mu)=0,\qquad \mu>0.$$
Hence, in this case, $\wt{a}$ extends to a $\kappa$-homogeneous function on $\rz^n\times V$ 
with $V=(\rz^n\times\rpbar)\setminus\{0\}$ by setting $a(x,0,\mu)=0$ for $\mu>0$. 
We shall implicitly identify $\wt{a}$  with its extension. 
\end{remark}

Let us also remark that
\begin{equation}\label{eq:inclusion} 
{S}^{d}(\rz^{n}\times\rpbar;E_0,E_1) \subseteq 
\wt{S}^{d,0}(\rz^{n}\times\rpbar;E_0,E_1)_{\bfkappa,\bfkappa}.
\end{equation} 

\subsection{Oscillatory integrals and calculus}\label{sec:oscillatory}

With a parameter-dependent operator valued symbol $a$ from one of the spaces introduced above, 
we associate the $\mu$-dependent family of continuous operators 
 $$\op(a)(\mu)=a(x,D;\mu):\scrS(\rz^n,E_0)\lra \scrS(\rz^n,E_1)$$
defined by  
 $$[\op(a)(\mu)u](x)=\int e^{ix\xi}a(x,\xi;\mu)\wh{u}(\xi)\,\dbar\xi,\qquad u\in\scrS(\rz^n,E_0).$$

The most efficient way to describe the calculus for such operator-families is based on the concept of 
oscillatory integrals in the spirit of \cite{Kuma}, but extended to Fr\`{e}chet space valued amplitude functions. 
For convenience of the reader and for applications in the sequel we give a short summary of 
this concept; for more details see \cite{SeilDiss}, \cite{GSS}. 

Let $E$ be a Fr\'echet space whose topology is described by a system of semi-norms $p_n$, $n\in\nz$. 
A smooth function $q(y,\eta):\rz^m\times\rz^m\to E$ is called an amplitude function with values in $E$, provided there exist sequences $(m_n)$ and $(\tau_n)$ such that 
 $$p_n\big(D^\alpha_\eta D^\beta_y q(y,\eta)\big)\preceq \spk{y}^{\tau_n}\spk{\eta}^{m_n}$$
for all $n$ and for all orders of derivatives. If $\chi(y,\eta)$ denotes a cut-off function with $\chi(0,0)=1$, the so-called oscillatory integral 
 $$\mathrm{Os}-\iint e^{-iy\eta}q(y,\eta)\,dy\dbar\eta:=
   \lim_{\eps\to0} \iint_{\rz^n\times\rz^n} e^{-iy\eta}\chi(\eps y,\eps\eta)q(y,\eta)\,dy\dbar\eta$$
exists and is independent on the choice of $\chi$. Note that for a continuous, $E$-valued function $f$ with compact support, $\iint f(y,\eta)\,dy d\eta$ is the unique element $e\in E$ such that 
$\spk{e^\prime,e}=\iint \spk{e^\prime,f(y,\eta)}\,dy d\eta$ for every functional $e^\prime\in E$. 
For simplicity of notation we shall simply write $\iint$ rather than $\mathrm{Os}-\iint$. 

\begin{example}\label{ex:Leibniz}
Given two parameter-dependent symbols $a_0$ and $a_1$, consider 
 $$q(y,\eta)=\big[(x,\xi)\mapsto a_1(x+y,\xi;\mu)a_0(x,\xi+\eta;\mu)\big],\qquad y,\eta\in\rz^n.$$ 
It is then straight-forward to verify the following$:$
\begin{itemize}
 \item[i$)$] If $a_j\in S^{d_j}_{1,0}(\rz^n\times\rpbar;E_j,E_{j+1})$, $j=0,1$, then 
  $q$ is an amplitude function with values in $S^{d_0+d_1}_{1,0}(\rz^n\times\rpbar;E_0,E_{2})$. 
 \item[ii$)$] If $a_j\in \wt{S}^{d_j,\nu_j}_{1,0}(\rz^n\times\rpbar;E_j,E_{j+1})_{\bfkappa,\bfkappa}$,   
  $j=0,1$, then $q$ is an amplitude function with values in $\wt{S}^{d_0+d_1,\nu_0+\nu_1}_{1,0}(\rz^n\times\rpbar;E_0,E_{2})_{\bfkappa,\bfkappa}$. 
\end{itemize}
Correspondingly, in each of these cases, we can define the so-called Leibniz-product of 
$a_0$ and $a_1$ by 
 $$(a_1\#a_0)(x,\xi;\mu)=\iint e^{-iy\eta}a_1(x+y,\xi;\mu)a_0(x,\xi+\eta;\mu)\,dy\dbar\eta.$$
The Leibniz-product corresponds to composition of operators, i.e., 
\begin{equation}\label{eq:Leibniz}
 \op(a_1\#a_0)(\mu)=\op(a_1)(\mu)\op(a_0)(\mu).
\end{equation}
\end{example} 
 
Another important operation for pseudo-differential operators is the so-called formal 
adjoint. 

\begin{definition}
We call $(E,H,\wt{E})$ a Hilbert-triple if the inner product of $H$ induces a non-degenerate sesquilinear pairing 
$E\times \wt{E}\to\cz$ that permits to identify the dual spaces of $E$ and $\wt{E}$ with $\wt{E}$ and $E$, 
respectively, and
 $$(\kappa_{\lambda}e,\wt{e})_H=(e,\wt{\kappa}_{\lambda}\wt{e})_H\qquad 
     \forall\;e,\,\wt{e},\,\lambda>0.$$
\end{definition}

\begin{example}
For every choice of $s,\delta,\gamma\in\rz$, both 
 $$(H^{s,\delta}(\rz_+),L^2(\rz_+),H^{-s,-\delta}_0(\rpbar)),\quad 
 (\calK^{s,\gamma}(\rz_+)^{\delta},L^2(\rz_+),\calK^{-s,-\gamma}(\rz_+)^{-\delta})$$ 
together with the standard group-action of Example $\ref{ex:standard-group-action}$ 
are Hilbert-triples. 
\end{example}

 If $(E_0,H_0,\wt{E}_0)$ and $(E_1,H_1,\wt{E}_1)$ are two Hilbert-triples and 
 $A:\scrS(\rz^n,E_0)\to\scrS(\rz^n,E_1)$, then denote by $A^*$ 
 the operator $\scrS(\rz^n,\wt{E}_1)\to\scrS(\rz^n,\wt{E}_0)$ defined by  
 $$(Au,v)_{L^2(\rz^n,H_1)}=(u,A^*v)_{L^2(\rz^n,H_0)},\qquad 
     u\in\scrS(\rz^n,E_0),\;v\in\scrS(\rz^n,\wt{E}_1).$$ 
This is the so-called formal adjoint of $A$. 

\begin{example}
Let $(E_0,H_0,\wt{E}_0)$ and $(E_1,H_1,\wt{E}_1)$ be two Hilbert-triples. 
	Given a symbol $a$, consider 
	$$q(y,\eta)=\big[(x,\xi)\mapsto a(x+y,\xi+\eta;\mu)^*\big],\qquad y,\eta\in\rz^n.$$ 
	It is then straight-forward to verify the following$:$
	\begin{itemize}
		\item[i$)$] If $a\in S^{d}_{1,0}(\rz^n\times\rpbar;E_0,E_{1})$ then 
		$q$ is an amplitude function with values in $S^{d}_{1,0}(\rz^n\times\rpbar;\wt{E}_1,\wt{E}_{0})$. 
		\item[ii$)$] If $a\in \wt{S}^{d,\nu}_{1,0}(\rz^n\times\rpbar;E_0,E_{1})_{\bfkappa,\bfkappa}$ then 
		$q$ is an amplitude function with values in 
		$\wt{S}^{d_0,\nu_0}_{1,0}(\rz^n\times\rpbar;\wt{E}_1,\wt{E}_{0})_{\bfkappa,\bfkappa}$. 
	\end{itemize}
	Correspondingly, in each of these cases, we can define the adjoint symbol of $a$ by 
	$$a^{(*)}(x,\xi;\mu)=\iint e^{-iy\eta}a(x+y,\xi+\eta;\mu)^*\,dy\dbar\eta.$$
The adjoint symbol corresponds to the formally adjoint operator, i.e.,   
\begin{equation}\label{eq:formal_adjoint}
 \op(a)(\mu)^*=\op(a^{(*)})(\mu).
\end{equation}
\end{example} 

To keep the exposition short we have focused on so-called left-symbols $a(x,\xi;\mu)$, 
only depending on the variable $x$. 
As in the standard theories of pseudodifferential operators, one can also 
introduce the corresponding classes of double-symbols $a(x,y,\xi;\mu)$ by substituting 
in the above definitions $x$ by $(x,y)$. The associated pseudodifferential operator is then given by
\begin{equation}\label{eq:double-symbol}
 \op(a)(\mu)u](x)=\iint e^{-iy\eta}a(x,x+y,\eta;\mu)u(x+y)\,dy\dbar\eta.
\end{equation}
Here, for fixed $\mu$, the integrand $a(x,y,\xi;\mu)u(x+y)$ can be viewed as an amplitude function 
with values in $\scrS(\rz^n,E_1)$. Actually, in this representation, one can also admit functions 
$u\in\scrC^\infty_b(\rz^n,E_0)$, since then the integrand is an amplitude function with values in 
$\scrC^\infty_b(\rz^n,E_1)$. Starting out from a double-symbol $a(x,y,\xi;\mu)$ there exist 
unique left- and right-symbols $a_L(x,\xi;\mu)$ and $a_R(y,\xi;\mu)$ such that 
 $$\op(a)(\mu)=\op(a_R)(\mu)=\op(a_L)(\mu).$$
There are explicit oscillatory-integral formulas for $a_L$ and $a_R$; 
for details we refer the reader to the literature.  
\section{Parameter-dependent symbols of finite regularity}\label{sec:finreg}

The following definition is motivated by Grubb's calculus of pseudodifferential 
operators with finite regularity. 

\begin{definition}\label{def:sdnu}
With $d,\nu\in\rz$ set
$$ {S}^{d,\nu}_{1,0}(\rz^{n}\times\rpbar;E_0,E_1):=  
 {S}^{d}_{1,0}(\rz^{n}\times\rpbar;E_0,E_1) 
 +\wt{S}^{d,\nu}_{1,0}(\rz^{n}\times\rpbar;E_0,E_1)_{\bfkappa,\bfkappa}.  
$$
Analogously define ${S}^{d,\nu}(\rz^{n}\times\rpbar;E_0,E_1)$. 
The corresponding spaces of $\kappa$-homogeneous symbols is
\begin{align*}
\begin{split}
{S}^{(d,\nu)}(\rz^{n}\times\rpbar;E_0,E_1):=&  
{S}^{(d)}(\rz^{n}\times\rpbar;E_0,E_1)+\\ 
&+\wt{S}^{(d,\nu)}(\rz^{n}\times\rpbar;E_0,E_1)_{\bfkappa,\bfkappa}.  
\end{split}
\end{align*} 
\end{definition}

All spaces in the previous definition are a non-direct sum of Fr\'{e}chet spaces, hence a Fr\`echet space itself. 
In general, we shall use notations like $a=a_0+\wt{a}$ if we represent an arbitrary symbol 
as the sum of two symbols. We shall say that such symbols have order $d$ and regularity $\nu$. 

\begin{remark}\label{rem:sdnue0e1}
	A symbol  
	$a\in {S}^{d,\nu}_{1,0}(\rz^{n}\times\rpbar;E_0,E_1)$ satisfies the estimates 
	\begin{align*}
	\|\kappa^{-1}_1(\xi,\mu)&\big\{D^\beta_{x}D^\alpha_{\xi}D^j_\mu
	a(x,\xi;\mu)\big\}\kappa_0(\xi,\mu)\|_{\scrL(E_0,E_1)}\\ 
	&\preceq \Big[\Big(\frac{\spk{\xi}}{\spk{\xi,\mu}}\Big)^{\nu-|\alpha|}+1\Big]\spk{\xi,\mu}^{d-|\alpha|-j}
	\end{align*}
	for every order of derivatives. Analogously for the space of $\kappa$-homogeneous functions
	where one needs to replace $\spk{\xi}$, $\spk{\xi,\mu}$ by $|\xi|$ and $|\xi,\mu|$, respectively. 
\end{remark}

Note that 
\begin{align*}
 {S}^{d,\nu}_{1,0}(\rz^{n}\times\rpbar;E_0,E_1)
  =\wt{S}^{d,\nu}_{1,0}(\rz^{n}\times\rpbar;E_0,E_1)_{\bfkappa,\bfkappa},\qquad \nu\le0
\end{align*}
and
\begin{align*}
\begin{split}
 {S}^{d,+\infty}_{1,0}(\rz^{n}\times\rpbar;E_0,E_1):=&
 \mathop{\mbox{\large$\cap$}}_{\nu\in\rz}{S}^{d,\nu}_{(\cl)}(\rz^{n}\times\rpbar;E_0,E_1)\\
 =& {S}^{d}_{1,0}(\rz^{n}\times\rpbar;E_0,E_1);
\end{split}
\end{align*}
in particular, infinite regularity corresponds to strong parameter-dependence. The corresponding fact is true for the classes of classical symbols. 

From now on, for convenience of notation, we shall drop $\rz^n\times\rpbar$ from notation and will use the short-hand notations 
\begin{align*} 
{S}^{d}_{1,0}(E_0,E_1)&:={S}^{d}_{1,0}(\rz^{n}\times\rpbar;E_0,E_1),\\
{S}^{d,\nu}_{1,0}(E_0,E_1)&:={S}^{d,\nu}_{1,0}(\rz^{n}\times\rpbar;E_0,E_1),\\ 
\wt{S}^{d,\nu}_{1,0}(E_0,E_1)_{\bfkappa,\bfkappa}
&:=\wt{S}^{d,\nu}_{1,0}(\rz^{n}\times\rpbar;E_0,E_1)_{\bfkappa,\bfkappa}\end{align*}
and similarly for the spaces of classical and $\kappa$-homogeneous functions.   


\subsection{Calculus: composition and formal adjoint} 

Putting together the results of Section \ref{sec:oscillatory} and the embedding \eqref{eq:inclusion}, 
the following theorems are easily verified$:$ 

\begin{theorem}\label{thm:comp02}
The Leibniz-product of symbols, respectively the composition of pseudodifferential operators, induces a bilinear continuous map 
	$$(a_1,a_0)\mapsto a_1\#a_0: 
	{S}^{d_1,\nu_1}_{1,0}(E_1,E_2)\times
	{S}^{d_0,\nu_0}_{1,0}(E_0,E_1)\lra 
	{S}^{d,\nu}_{1,0}(E_0,E_2)$$ 
	and analogously for classical symbols, where 
	$$d=d_0+d_1,\qquad \nu=\min(\nu_0,\nu_1,\nu_0+\nu_1).$$
\end{theorem}

In case $a_j\in \wt{S}^{d_j,\nu_j}_{1,0}(E_j,E_{j+1})_{\bfkappa,\bfkappa}$ for at least one of the 
values $j=0$ or $j=1$, then $a_1\#a_0\in\wt{S}^{d,\nu}_{1,0}(E_0,E_2)_{\bfkappa,\bfkappa}$ $($analogously for classical symbols$)$. In this sense, the $\wt{S}$-class is a two-sided ideal. 

Using the explicit formula for the Leibniz-product and the standard technique of Taylor expansion, 
one sees that, for every $N\in\nz_0$, 
\begin{equation}\label{eq:comp_exp}
	r_N(a_1,a_0):=a_1\#a_0
	  -\sum_{|\alpha|=0}^{N-1}\frac{1}{\alpha!}(\partial_\xi^\alpha a_1)(D^\alpha_x a_0)
	  \;\in\; {S}^{d-N,\nu-N}_{1,0}(E_0,E_2)
\end{equation}
with continuous dependence on $a_0$ and $a_1$ $($analogously for classical symbols$)$. 

\begin{theorem}\label{thm:adj02}
Let $(E_0,H_0,\wt{E}_0)$ and $(E_1,H_1,\wt{E}_1)$ be two Hilbert-triples. 
Taking the formal adjoint induces a linear continuous map 
 $$a\mapsto a^{(*)}:  {S}^{d,\nu}_{1,0}(E_0,E_1)\lra {S}^{d,\nu}_{1,0}(\wt{E}_1,\wt{E}_0).$$ 
Moreover, for every $N\in\nz_0$ and with continuous dependence on $a$, 
	$$r_N^{(*)}(a):=a^{(*)}-\sum_{|\alpha|=0}^{N-1}\frac{1}{\alpha!}
	  \partial_\xi^\alpha D^\alpha_x a^*\;\in\;{S}^{d-N,\nu-N}_{1,0}(\wt{E}_1,\wt{E}_0).$$ 
The analogous statements hold true for classical symbols. 
\end{theorem}

\subsection{The homogeneous principal symbol} 

Given $a=a_0+\wt{a}\in {S}^{d,\nu}(E_0,E_1)$ we define 
\begin{align}\label{eq:principalsymb}
\begin{split}
 \sigma^{d,\nu}(a)(x,\xi;\mu)
 &=a_0^{(d)}(x,\xi;\mu)+\wt{a}^{(d,\nu)}(x,\xi;\mu)\\
 &=\lim_{\lambda\to+\infty}
   \lambda^{-d}\kappa_{1,\lambda}^{-1}a(x,\lambda\xi;\lambda\mu)\kappa_{0,\lambda},
   \qquad\xi\not=0.
\end{split}
\end{align}
$($recall that in case of positive regularity $\nu>0$, the principal symbol extends by continuity 
to all $(\xi,\mu)\not=0$, cf. Remark \ref{rem:extension}$)$.  
Due to \eqref{eq:comp_exp}, the principal symbol behaves multiplicatively under composition, 
 $$\sigma^{d,\nu}(a_1\#a_0)=\sigma^{d_1,\nu_1}(a_1)\sigma^{d_0,\nu_0}(a_0).$$

\begin{proposition}\label{prop:principal-map}
Let $\chi(\xi,\mu)$ and $\wt\chi(\xi)$ be two $0$-excision functions. Then  
 $$a_0^{(d)}+\wt{a}^{(d,\nu)}\mapsto \chi a_0^{(d)}+\wt{\chi}\,\wt{a}^{(d,\nu)}$$
induces a well-defined and surjective map 
 $${S}^{(d,\nu)}(E_0,E_1)\lra 
   {S}^{d,\nu}(E_0,E_1)\big/
   {S}^{d-\infty,\nu-\infty}(E_0,E_1).$$
\end{proposition}
\begin{proof}
After multiplication with $|\xi,\mu|^{-d}$ we may assume w.l.o.g. that $d=0$.
Then it suffices to show that 
$\chi a_0^{(0)}+\wt{\chi}\,\wt{a}^{(0,\nu)}\in \wt{S}^{-\infty,\nu-\infty}(E_0,E_1)_{\bfkappa,\bfkappa}$ 
provided $a_0^{(0)}+\wt{a}^{(0,\nu)}=0$. Since changing the cut-off functions results in a smoothing remainder, we may assume that $\chi=0$ on $|\xi,\mu|\le1$ and 
$\chi=1$ on $|\xi,\mu|\ge2$ as well as $\wt\chi=0$ on $|\xi|\le1/2$ and 
$\wt\chi=1$ on $|\xi|\ge1$. 

Let $\chi_1(\xi,\mu)$ be another $0$-excision function, vanishing for $|\xi,\mu|\le2$ and being $1$ for $|\xi,\mu|\ge3$. 
Since $\wt{a}^{(0,\nu)}=-a_0^{(0)}$ we have 
 $$\chi a_0^{(0)}+\wt{\chi}\,\wt{a}^{(0,\nu)}=(\chi-\wt\chi)a_0^{(0)}
   =(1-\chi_1)(\chi-\wt\chi)a_0^{(0)}+\chi_1(\chi-\wt\chi)a_0^{(0)}.$$
The first term on the right-hand side is a smooth compactly supported function, hence is 
regularizing. The second term is supported in the strip-like region 
$\Sigma:=\{(\xi,\mu)\mid |\xi,\mu|\ge 2,\;|\xi|\le1\}$. 

In case $\nu\le0$ it is clear that the second term is regularizing. So we assume $\nu>0$. 

Let $N$ be the largest integer strictly smaller than $\nu$. Using the symbol estimates in 
$\wt{S}^{(0,\nu)}(E_0,E_1)_{\bfkappa,\bfkappa}$ and \eqref{eq:M-kappa} we find that  
 $$\|D^\alpha_\xi a_0^{(0)}(x,\xi;\mu)\|_{\scrL(E_0,E_1)}\preceq 
    |\xi|^{\nu-|\alpha|}|\xi,\mu|^{-\nu}\max\{|\xi,\mu|,|\xi,\mu|^{-1}\}^M,\qquad 
   \xi\not=0,$$
with a suitable $M\ge 0$. Thus these derivatives vanish for $\xi=0$ and $\mu>0$ provided $|\alpha|\le N$.
 By Taylor expansion we conclude 
 $$a_0^{(0)}(x,\xi;\mu)=\sum_{|\alpha|=N+1}\frac{N+1}{\alpha!}\xi^\alpha
   \int_0^1(1-t)^N(\partial^\alpha_\xi a_0^{(0)})(x,t\xi;\mu)\,dt.$$
Now observe that 
 $$\frac{1}{2}\le\frac{\mu}{\spk{\mu}}\le\frac{\mu}{|\xi,\mu|}
   \le\frac{|t\xi,\mu|}{|\xi,\mu|}\le1,\qquad\forall\;(\xi,\mu)\in\Sigma,\;0\le t\le1.$$
Hence 
 $$\|\kappa^{\pm1}_{\ell,|t\xi,\mu|/|\xi,\mu|}\|_{\scrL(E_\ell)}\preceq 1,
   \qquad\forall\;(\xi,\mu)\in\Sigma,\;0\le t\le1,$$ 
and we can estimate, for $(\xi,\mu)\in\Sigma$, 
\begin{align*}
 \|\kappa^{-1}_1(\xi,\mu)&\big\{D^\beta_{x}D^\alpha_{\xi}D^j_\mu
 a_0^{(0)}(x,\xi;\mu)\big\}\kappa_0(\xi,\mu)\|_{\scrL(E_0,E_1)}\\
 &\preceq \spk{\xi}^{-L}\int_0^1 |t\xi,\mu|^{-(N+1)-j}\,dt 
  \preceq \spk{\xi}^{-L}\spk{\mu}^{-\nu-j}
\end{align*}
with arbitrary $L\ge 0$ $($recall that $|\xi|\le1$ in $\Sigma)$. 
It follows that $\chi_1(\chi-\wt\chi)a_0^{(0)}$ belongs to 
$\wt{S}^{-\infty,\nu-\infty}(E_0,E_1)_{\bfkappa,\bfkappa}$. This completes the proof. 
\end{proof}

\begin{corollary}
Let $a\in {S}^{d,\nu}(E_0,E_1)$ with $\sigma^{d,\nu}(a)=0$. 
Then $a\in {S}^{d-1,\nu-1}(E_0,E_1)$. 
In other words, we have the short exact sequence 
 $$0\lra {S}^{d-1,\nu-1}(E_0,E_1)\hookrightarrow 
   {S}^{d,\nu}(E_0,E_1)\xrightarrow{\sigma^{d,\nu}}
   {S}^{(d,\nu)}(E_0,E_1)\lra0.$$
\end{corollary}

\subsection{Ellipticity and parametrix construction} 

In the present set-up, one can introduce a concept of ellipticity only in case $\nu>0$. Again, after multiplication by  
$[\xi,\mu]^{-d}$, we can focus on the case of symbols of order $d=0$. 

Below we shall frequently use the following fact$:$ If $u:\Omega_y\to\scrL(E_0,E_1)$ is a smooth function which is pointwise invertible on $\Omega\subset\rz^m$, then $D^\alpha_y u(y)^{-1}$ is a linear combination of terms of the form 
\begin{equation}\label{eq:deriv_inv}
[u^{-1}D^{\alpha_1}_y u]\cdots [u^{-1}D^{\alpha_\ell}_y u] u^{-1},
\qquad \ell\ge 1,\quad \alpha_1+\ldots+\alpha_\ell=\alpha. 
\end{equation} 

\begin{lemma}\label{lem:inv01}
	Let $\nu>0$ and 
	$\wt{a}\in \wt{S}^{(0,\nu)}(E_0,E_1)_{\bfkappa,\bfkappa}$. Let 
	$\chi\in\scrC^\infty(\sz^{n}_+)$ with $0\le\chi\le1$ and define 
	$$\wt{a}_\chi(x,\xi;\mu)
	=\chi\Big(\frac{(\xi,\mu)}{|\xi,\mu|}\Big)
	\wt{a}(x,\xi;\mu).$$
	Then there exists a neighborhood $K\subset\sz^{n}_+$ of the north-pole $(0,1)$ such that 
	if $\chi$ is supported in $K$ then $(1+a_\chi )(x,\xi;\mu)$ is invertible whenever $(\xi,\mu)\not=0$. Moreover,  
	there exists a $\wt{b}\in \wt{S}^{(0,\nu)}(E_0,E_1)_{\bfkappa,\bfkappa}$
	such that 
	$$(1+a_\chi)^{-1}=1+\wt b.$$ 
\end{lemma}
\begin{proof}
	Note that $a_\chi\in \wt{S}^{(0,\nu)}(E_0,E_1)_{\bfkappa,\bfkappa}$. 
	Choose $K$ in such a way that $\|\wt{a}(x,\xi;\mu)\|\preceq |\xi|^\nu\le 1/2$ 
	for all $x$ and all $(\xi,\mu)\in K$. Hence $1+a_\chi$ is invertible on 
	$\rz^{n}\times\sz^{n}$ and $(1+a_\chi)^{-1}$ is uniformly bounded. 
	By $\kappa$-homogeneity it follows that $(1+a_\chi)^{-1}$ exists everywhere with 
	$$\|\kappa^{-1}_{0,|\xi,\mu|}
	(1+a_\chi(x,\xi;\mu))^{-1}
	\kappa_{1,|\xi,\mu|}\|_{\scrL(E_1,E_0)} \preceq 1$$
	whenever $(\xi,\mu)\not=0$. From \eqref{eq:deriv_inv} it follows that 
	$(1+a_\chi)^{-1}\in\wt{S}^{(0,0)}(E_0,E_1)_{\bfkappa,\bfkappa}$. But then 
	$$(1+a_\chi)^{-1}=1-a_\chi+a_\chi(1+a_\chi)^{-1}a_\chi=:1+\wt{b}$$
	gives the desired result.  
\end{proof}

\begin{proposition}\label{prop:inverse_principal}
	Let $\nu>0$ and $a\in{S}^{(0,\nu)}(E_0,E_1)$. 
	Assume that $a(x,\xi;\mu)$ is invertible whenever $(\xi,\mu)\not=0$ and that 
	$$\|a(x,\xi;\mu)^{-1}\|_{\scrL(E_1,E_0)} \preceq 1
	  \qquad \forall\;x\in\rz^n\quad\forall\;|\xi,\mu|=1. $$
	Then $a^{-1}\in{S}^{(0,\nu)}(E_1,E_0)$. 
\end{proposition}
\begin{proof}
	Write $a=a_0+\wt{a}$. Obviously, $a^{-1}$ is $\kappa$-homogeneous of degree $0$. Thus the assumed estimate is equivalent to 
	$$\|\kappa^{-1}_{0,|\xi,\mu|}a(x,\xi;\mu)^{-1}\kappa_{1,|\xi,\mu|}
	\|_{\scrL(E_1,E_0)} \preceq 1\qquad\forall\;x\in\rz^n\quad\forall\;(\xi,\mu)\not=0.$$
	By chain rule $($cf. \eqref{eq:deriv_inv}$)$ it follows that 
	\begin{align*}
	\|\kappa^{-1}_{0,|\xi,\mu|}\{D^\beta_{x}D^\alpha_{\xi^\prime}D^j_\mu a^{-1}
	(x,\xi;\mu)\}\kappa_{1,|\xi,\mu|}\|_{\scrL(E_1,E_0)}
	&\preceq 
	\Big[\Big(\frac{|\xi|}{|\xi,\mu|}\Big)^{\nu-|\alpha|}+1\Big]
	|\xi,\mu|^{-|\alpha|-j}
	\end{align*}
	Hence, if $\wh{\chi}(\xi,\mu)={\chi}((\xi,\mu)/|\xi,\mu|)$ 
	denotes the homogeneous extension of degree $0$ 
	of an arbitrary smooth function $\chi\in\scrC^{\infty}(\sz^{n}_+)$ which is identically 
	to $1$ in a neighborhood of the north-pole, then 
	$(1-\wh\chi)a^{-1}$ belongs to $S^{(0)}(E_1,E_0)$, since $|\xi|\sim|\xi,\mu|$ on the support of $1-\wh\chi$.  
	
	Since $a(x,0;1)=a_0(x,0;1)$, there exists a $C\ge 0$ such that 
	$\|a_0(x,0;1)^{-1}\|_{\scrL(E_1,E_0))}\le C$ for every $x$. By Taylor formula, 
	$$\|a_0(x,r\phi;\sqrt{1-r^2})-a_0(x,0;1)\|_{\scrL(E_0,E_1))}\preceq r$$
	uniformly in $x$ and $\phi\in\sz^{n-1}$. Hence there exists 
	a neighborhood $K\subset\sz^{n}_+$ of the north-pole $(0,1)$ such that 
	$a_0(x,\xi;\mu)^{-1}$ exists and is uniformly bounded in $\rz^{n}\times K$. 
	Thus if we choose $\chi$ to be compactly supported in $K$ and use chain-rule 
	$($cf. \eqref{eq:deriv_inv}$)$, we obtain that $\wh{\chi}a_0^{-1}$ belongs to 
	$S^{(0)}(E_1,E_0)$. 
	
	Now let $\chi_1\in\scrC^\infty(\sz^{n}_+)$ be also supported in $K$, $0\le \chi_1\le 1$, and $\chi_1\equiv 1$ in a neighborhood of the support of $\chi$. Then we can write 
	$$\wh{\chi}a^{-1}=\wh{\chi}(a_0+\wt{a})^{-1}=\wh{\chi}a_0^{-1}(1+\wh\chi_1a_0^{-1}\wt{a})^{-1},$$
	where $\wh\chi_1$ is the extension of $\chi_1$ by homogeneity of degree $0$. 
	Since $\wh\chi_1a_0^{-1}\wt{a}$ belongs to 
	$\wt{S}^{(0,\nu)}(E_0,E_1)_{\bfkappa,\bfkappa}$, we may choose $\chi_1$ 
	$($and $\chi)$ to be supported so closely to the north-pole that Lemma \ref{lem:inv01} applies, 
	i.e., $(1+\wh\chi_1a_0^{-1}\wt{a})^{-1}=1+\wt{b}$ with a symbol 
	$\wt{S}^{(0,\nu)}(E_1,E_0)_{\bfkappa,\bfkappa}$. Altogether we find 
	$$a^{-1}=[(1-\wh\chi)a^{-1}+\wh\chi a_0^{-1}]+\wh\chi a_0^{-1}\wt{b}.$$
	This is the desired representation of $a^{-1}$. 
\end{proof}

\begin{definition}
A symbol $a\in S^{0,\nu}(E_0,E_1)$, $\nu>0$, is called elliptic if  
its principal symbol $\sigma^{(0,\nu)}(a)(x,\xi;\mu)$ is invertible whenever $(\xi,\mu)\not=0$ and  
$$\|\sigma^{(0,\nu)}(a)(x,\xi;\mu)^{-1}\|_{\scrL(E_1,E_0)} \preceq 1\qquad  \forall\;x\in\rz^n\quad\forall\;|\xi,\mu|=1. $$
\end{definition}

\begin{theorem}
	Let $a\in S^{0,\nu}(E_0,E_1)$, $\nu>0$, be elliptic. 
	Then $a$ has a parametrix $b\in S^{0,\nu}(E_1,E_0)$, i.e.,  
	$$1-a\#b\in S^{-\infty,\nu-\infty}(E_1,E_1),\qquad 
	  1-b\#a\in S^{-\infty,\nu-\infty}(E_0,E_0).$$
\end{theorem}
\begin{proof}
Due to the previous proposition, we can write $\sigma^{(0,\nu)}(a)^{-1}=b^{(0)}+\wt{b}^{(0,\nu)}$ with 
the obvious meaning of notation. 
Define $b_0=\chi_0(\xi,\mu)b^{(0)}+\chi_1(\xi)\wt{b}^{(0,\nu)}$ with $0$-excision functions $\chi_0$ and $\chi_1$. Then $b_0\in S^{0,\nu}(E_1,E_0)$ and $r_0:=1-ab_0\in S^{-1,\nu-1}_\cl(E_1,E_1)$, since $r_0$ has vanishing principal symbol. By the standard Neumann series argument we find a $b\in S^{0,\nu}(E_1,E_0)$ which is a right-parametrix of $a$. Analogously, we find a left-parametrix. It differs from $b$ by a regularizing term, hence $b$ is both left- and right-parametrix. 
\end{proof}

\section{Poisson operators of Grubb's class}\label{sec:poisson}

In this section we discuss the main example that underlies our approach to describe singular Green operators, as it will be presented in the following Section \ref{sec:green02}. 
 
To this end, let us recall the definition of Poisson operators due to \cite{Grub}. 
A parameter-dependent Poisson-operator of order $d+\frac{1}{2}$ and regularity $\nu$ 
is of the form 
\begin{equation}\label{eq:poisson-operator}
 [K(\mu)u](x^\prime,x_n)=\int e^{ix^\prime\xi^\prime}
  k(x^\prime,x_n,\xi^\prime;\mu)\wh{u}(\xi^\prime)\,\dbar\xi^\prime,
 \qquad u\in\scrS(\rz^{n-1}),
\end{equation}
where the so-called symbol-kernel $k(x^\prime,x_n,\xi^\prime;\mu)$ is a smooth function 
satisfying the estimates 
\begin{align}\label{eq:poisson-kernel}
\begin{split} 
\|D^{\beta}_{x^\prime}D^{\alpha}_{\xi^\prime}D^j_\mu x_n^\ell D_{x_n}^{\ell^\prime}
   k\|_{L_2(\rz_+)}
   \preceq& \Big[\Big(\frac{\spk{\xi^\prime}}{\spk{\xi^\prime,\mu}}\Big)^{\nu-[\ell-\ell^\prime]_+-|\alpha|}+1\Big]\spk{\xi^\prime,\mu}^{d-\ell+\ell^\prime-|\alpha|-j}
\end{split}
\end{align}
for every order of derivatives and every $\ell\in\nz_0$. Here, for arbitrary $r\in\rz$, 
$$r_+=\max\{0,r\},\qquad r_-=\max\{0,-r\},\qquad r=r_+-r_-.$$  
Obviously, for every fixed $\mu$, a Poisson operator induces a map 
 $$K(\mu):\scrS(\rz^{n-1})\lra \scrS(\rz^n_+):=\scrS(\rz^n)|_{\rz^n_+},$$
i.e., maps functions defined on the boundary of the half-plane $\rz^n_+$ to functions defined on the half-space. 

We shall need the following definition$:$ 

\begin{definition}
	Let $E$ be a Fr\'echet space. 
	Then $\mathfrak{S}^{d}_{1,0}(\rz^{n-1}\times\rpbar;E)=:\mathfrak{S}^{d}_{1,0}(E)$  
	consists of all smooth $E$-valued functions 
	$\mathfrak{a}(x^\prime,\xi^\prime;\mu)$ satisfying the uniform estimates 
	$$p(D^{\beta}_{x^\prime}D^\alpha_{\xi^\prime}D^j_\mu \mathfrak{a}(x^\prime,\xi^\prime;\mu))\preceq \spk{\xi^\prime,\mu}^{d-|\alpha|-j}$$
	for every continuous semi-norm $p(\cdot)$ of $E$ and every order of derivatives. Instead, the estimates 
	$$p(D^{\beta}_{x^\prime}D^\alpha_{\xi^\prime}D^j_\mu \mathfrak{a}(x^\prime,\xi^\prime;\mu))\preceq 
	\spk{\xi^\prime}^{\nu-|\alpha|} \spk{\xi^\prime,\mu}^{d-\nu-j}$$	
	define the space 
	$\wt{\mathfrak{S}}^{d,\nu}_{1,0}(\rz^{n-1}\times\rpbar;E)=:\wt{\mathfrak{S}}^{d,\nu}_{1,0}(E)$. 
\end{definition}


\subsection{Strongly parameter-dependent Poisson operators}\label{sec:strongly-poisson}

First let us discuss those Poisson operators that depend strongly on the parameter, i.e., have 
infinite regularity. In case of order $d+\frac{1}{2}$, the corresponding symbol-kernels 
are characterized by the estimates 
 $$\|D^{\beta}_{x^\prime}D^{\alpha}_{\xi^\prime}D^j_\mu x_n^\ell D_{x_n}^{\ell^\prime}k\|_{L^2(\rz_+)}
\preceq \spk{\xi^\prime,\mu}^{d-\ell+\ell^\prime-|\alpha|-j}.$$
For such operators the following characterization is known, cf. equations (3) and (6) in the 
proof of \cite[Theorem 2.1.19]{ScSc1}, for instance. 

\begin{theorem}\label{thm:poisson2nd}
A smooth function $k(x^\prime,x_n,\xi^\prime;\mu)$ is a strongly parameter-dependent Poisson symbol-kernel of order $d+\frac{1}{2}$ and infinite regularity if, and only if, it has the form 
 $$k(x^\prime,x_n,\xi^\prime;\mu)
     =\spk{\xi^\prime,\mu}^{1/2}\mathfrak{k}(x^\prime,\xi^\prime,\spk{\xi^\prime,\mu}x_n;\mu),$$
with 
 $$\mathfrak{k}(x^\prime,\xi^\prime,x_n;\mu)\in \mathfrak{S}^{d}_{1,0}(\rz^{n-1}\times\rpbar;\scrS(\rz_+)).$$
\end{theorem}

Observe, for purposes below, that in this case we can also write  
 $$k(x^\prime,x_n,\xi^\prime;\mu) =\spk{\xi^\prime}^{1/2}\wt{\mathfrak{k}}(x^\prime,\xi^\prime,\spk{\xi^\prime}x_n;\mu)$$ 
with 
$$\wt{\mathfrak{k}}(x^\prime,\xi^\prime,x_n;\mu)
=\Big(\frac{\spk{\xi^\prime,\mu}}{\spk{\xi^\prime}}\Big)^{1/2}
\mathfrak{k}\Big(x^\prime,\xi^\prime,\frac{\spk{\xi^\prime,\mu}}{\spk{\xi^\prime}}x_n;\mu\Big)
\in \wt{\mathfrak{S}}^{d,0}_{1,0}(\rz^{n-1}\times\rpbar;\scrS^0(\rz_+)) .$$

\subsection{Weakly parameter-dependent Poisson operators}\label{sec:weakly-poisson}

Returning to the defining estimate \eqref{eq:poisson-kernel}, strongly parameter-dependent Poisson symbol-kernels 
correspond to the second summand on the right-hand side. 
Those corresponding to the first summand, i.e., that satisfy the estimates 
\begin{equation}\label{eq:estimate0}
\|D^\beta_{x^\prime}D^{\alpha}_{\xi^\prime}D^j_\mu x_n^\ell D_{x_n}^{\ell^\prime}\wt{k}\|_{L_2(\rz_+)}
\preceq\Big(\frac{\spk{\xi^\prime}}{\spk{\xi^\prime,\mu}}\Big)^{\nu-[\ell-\ell^\prime]_+-|\alpha|}
\spk{\xi^\prime,\mu}^{d-\ell+\ell^\prime-|\alpha|-j}
\end{equation}
we shall refer to as weakly parameter-dependent Poisson symbol-kernels. 
They can be characterized as follows$:$

\begin{theorem}
A smooth function $\wt{k}(x^\prime,x_n,\xi^\prime;\mu)$ is a weakly parameter-dependent Poisson symbol-kernel 	of order $d+\frac{1}{2}$ and regularity $\nu$ if, and only if, there exist 
	\begin{align*}
	\wt{\mathfrak{k}}_1(x^\prime,\xi^\prime,\mu;x_n)&\in \wt{\mathfrak{S}}^{d,\nu}_{1,0}
	(\rz^{n-1}\times\rpbar;\scrS^0(\rz_+)),\\
	\wt{\mathfrak{k}}_2(x^\prime,\xi^\prime,\mu;x_n)&\in \wt{\mathfrak{S}}^{d,\nu}_{1,0}
	(\rz^{n-1}\times\rpbar;H^\infty(\rz_+))
	\end{align*}
such that
	$$\wt{k}(x^\prime,x_n,\xi^\prime;\mu)
	=\spk{\xi^\prime}^{1/2}\,\wt{\mathfrak{k}}_1(x^\prime,\xi^\prime,\mu;\spk{\xi^\prime}x_n)
	=\spk{\xi^\prime,\mu}^{1/2}\,\wt{\mathfrak{k}}_2(x^\prime,\xi^\prime,\mu;\spk{\xi^\prime,\mu}x_n).$$
\end{theorem}
\begin{proof}
	For simplicity of notation let us assume independence of the $x^\prime$-variable; the general case is verified 
	in the same way. 	
	By multiplication of $k$ with $\spk{\xi^\prime,\mu}^{\nu-d}\spk{\xi^\prime}^{-\nu}$ 
	we may assume w.l.o.g. that $\nu=d=0$. Thus let us start out from the estimates 
	\begin{align}\label{eq:estimate1}
	\begin{split} 
	\|D^{\alpha}_{\xi^\prime}D^j_\mu x_n^\ell D_{x_n}^{\ell^\prime}\wt{k}\|_{L^2(\rz_+)}
	&\preceq
	\begin{cases}
	\spk{\xi^\prime}^{\ell^\prime-\ell}\spk{\xi^\prime}^{-|\alpha|}\spk{\xi^\prime,\mu}^{-j}
	&\quad:\ell\ge\ell^\prime\\
	\spk{\xi^\prime,\mu}^{\ell^\prime-\ell}\spk{\xi^\prime}^{-|\alpha|}\spk{\xi^\prime,\mu}^{-j}
	&\quad:\ell<\ell^\prime
	\end{cases}.
	\end{split}
	\end{align}
	Now let us define 
	$$\wt{\mathfrak{k}}_1(\xi^\prime,\mu;x_n):=\spk{\xi^\prime}^{-1/2}\wt{k}(\xi^\prime,\mu;\spk{\xi^\prime}^{-1}x_n).$$
	Using chain rule it is not difficult to see that the estimates \eqref{eq:estimate1} are equivalent 
	to 
	\begin{align}\label{eq:estimate2}
	\|D^{\alpha}_{\xi^\prime}D^j_\mu x_n^\ell D_{x_n}^{\ell^\prime}\wt{\mathfrak{k}}_1\|_{L^2(\rz_+)}
	\preceq 
	\begin{cases}
	\spk{\xi^\prime}^{-|\alpha|}\spk{\xi^\prime,\mu}^{-j}
	&\quad:\ell\ge\ell^\prime\\
	\left(\frac{\spk{\xi^\prime,\mu}}{\spk{\xi^\prime}}\right)^{\ell^\prime-\ell}
	\spk{\xi^\prime}^{-|\alpha|}\spk{\xi^\prime,\mu}^{-j}
	&\quad:\ell<\ell^\prime
	\end{cases}
	\end{align}
	$($for all choices of indices $\alpha,j,\ell,\ell^\prime)$. 
	Now we argue that these estimates in turn are equivalent to 
	\begin{align}\label{eq:estimate3}
	\|D^{\alpha}_{\xi^\prime}D^j_\mu x_n^\ell D_{x_n}^{\ell^\prime}\wt{\mathfrak{k}}_1\|_{L^2(\rz_+)}
	\preceq \spk{\xi^\prime}^{-|\alpha|}\spk{\xi^\prime,\mu}^{-j},\qquad\ell\ge\ell^\prime,
	\end{align}
	and
	\begin{align}\label{eq:estimate4}
	\|D^{\alpha}_{\xi^\prime}D^j_\mu D_{x_n}^{\ell^\prime}\wt{\mathfrak{k}}_1\|_{L^2(\rz_+)}
	\preceq 
	\left(\frac{\spk{\xi^\prime,\mu}}{\spk{\xi^\prime}}\right)^{\ell^\prime}
	\spk{\xi^\prime}^{-|\alpha|}\spk{\xi^\prime,\mu}^{-j},
	\qquad\ell^\prime\ge 0. 
	\end{align}
	Clearly, \eqref{eq:estimate2} implies both \eqref{eq:estimate3} and \eqref{eq:estimate4}. 
	Vice versa, let $0<\ell<\ell^\prime$ be given. 
	Observe that, by H\"older inequality, 
	$$\|uv\|_{L^2(\rz_+)}\le \||u|^pv\|_{L^2(\rz_+)}^{1/p}\|v\|_{L^2(\rz_+)}^{1/q},
	\qquad {\frac{1}{p}}+{\frac{1}{q}}=1.$$ 
	We apply this inequality with $u=x^\ell_n$, $v=D^{\alpha}_{\xi^\prime}D^j_\mu 
	D_{x_n}^{\ell^\prime}\wt{\mathfrak{k}}_1$ 
	and $p=\ell^\prime/\ell$ and find 
	\begin{align*}
	\|D^{\alpha}_{\xi^\prime}D^j_\mu x_n^\ell D_{x_n}^{\ell^\prime}\wt{\mathfrak{k}}_1\|_{L^2(\rz_+)}
	\le\|D^{\alpha}_{\xi^\prime}D^j_\mu x_n^{\ell^\prime}    
	D_{x_n}^{\ell^\prime}\wt{\mathfrak{k}}_1\|_{L^2(\rz_+)}^{\ell/\ell^\prime}
	\|D^{\alpha}_{\xi^\prime}D^j_\mu D_{x_n}^{\ell^\prime}\wt{\mathfrak{k}}_1\|_{L^2(\rz_+)}^{1-\ell/\ell^\prime}.  
	\end{align*}
	Inserting the estimates from \eqref{eq:estimate3} and \eqref{eq:estimate4} yields 
	the second estimate in \eqref{eq:estimate2}. 
	
	If we define
	$$\wt{\mathfrak{k}}_2(\xi^\prime,\mu;x_n)
	=\left(\frac{\spk{\xi^\prime}}{\spk{\xi^\prime,\mu}}\right)^{1/2} 
	\wt{\mathfrak{k}}_1\Big(\xi^\prime,\mu;\frac{\spk{\xi^\prime}}{\spk{\xi^\prime,\mu}}x_n\Big)
	=\spk{\xi^\prime,\mu}^{-1/2}\wt{k}(\xi^\prime,\mu;\spk{\xi^\prime,\mu}^{-1}x_n),$$
	then \eqref{eq:estimate4} is equivalent to 
	\begin{align}\label{eq:estimate5}
	\|D^{\alpha}_{\xi^\prime}D^j_\mu D_{x_n}^{\ell^\prime}\wt{\mathfrak{k}}_2\|_{L^2(\rz_+)}
	\preceq 
	\spk{\xi^\prime}^{-|\alpha|}\spk{\xi^\prime,\mu}^{-j},
	\qquad\ell^\prime\ge 0. 
	\end{align}
	
	Summing up, we have verified that the assumption \eqref{eq:estimate1} is equivalent 
	to the validity of the estimates \eqref{eq:estimate3} and \eqref{eq:estimate5}.  
	Now recall that 
	\begin{align*}
	u\in\scrS^0(\rz_+)
	&\iff \omega u\in\scrH^{\infty,0}(\rz_+)\text{ and } (1-\omega) u\in\scrS(\rz)\\
	&\iff \|t^\ell D_t^{\ell^\prime} u\|_{L^2(\rz_+)}<+\infty\quad\forall\;\ell\ge\ell^\prime;  
	\end{align*}
	moreover, the semi-norms $\|t^\ell D_t^{\ell^\prime} u\|_{L^2(\rz_+)}$, 
	$\ell\ge\ell^\prime\in\nz_0$, define the topology of $\scrS^0(\rz_+)$. 
	Hence the estimates in \eqref{eq:estimate3} are equivalent to 
	$\wt{\mathfrak{k}}_1\in \wt{\mathfrak{S}}^{0,0}(\rz^{n-1}\times\rpbar;\scrS^0(\rz_+))$.  
	Similarly, recalling that 
	$$v\in H^\infty(\rz_+)\iff 
	\|D_t^{\ell^\prime} v\|_{L^2(\rz_+)}<+\infty\quad\forall\;\ell^\prime\in\nz_0,$$
	the estimates in \eqref{eq:estimate5} are equivalent to 
	$\wt{\mathfrak{k}}_2\in \wt{\mathfrak{S}}^{0,0}(\rz^{n-1}\times\rpbar;H^\infty(\rz_+))$. 
	This finishes the proof of the theorem.  
\end{proof}	

\subsection{Poisson symbol-kernels viewed as operator-valued symbols}\label{sec:opval-poisson}

We shall identify a Poisson-kernel $k(x^\prime,x_n,\xi^\prime;\mu)$ with the function 
 $$k:\rz^{n-1}_{x^\prime}\times\rz^{n-1}_{\xi^\prime}\times\rpbar
   \lra \scrL(\cz,\scrS(\rz_+)),\qquad c\mapsto k(x^\prime,\cdot,\xi^\prime;\mu)c.$$
Since $\scrS(\rz_+)$ is embedded both in $H^{s,\delta}(\rz_+)$ and $\scrK^{s,0}(\rz_+)^\delta$,  
\begin{align}\label{eq:xyz}
\begin{split}
 k\in &\mathop{\mbox{\large$\cap$}}_{s,\delta\in\rz}
 \scrC^\infty\big(\rz^{n-1}\times\rz^{n-1}\times\rpbar,\scrL(\cz,H^{s,\delta}(\rz_+))\big),\\
 k\in &\mathop{\mbox{\large$\cap$}}_{s,\rho\in\rz} 
 \scrC^\infty\big(\rz^{n-1}\times\rz^{n-1}\times\rpbar,\scrL(\cz,\scrK^{s,0}(\rz_+)^\rho)\big) 
\end{split}
\end{align}

\begin{theorem}\label{thm:charact-strongly}
A function $k$ from \eqref{eq:xyz} corresponds to a strongly parameter dependent Poisson symbol-kernel 
of order $d+\frac{1}{2}$ if, and only if, 
\begin{equation}\label{eq:abc}
 \|\kappa^{-1}(\xi^\prime,\mu) D^\alpha_{\xi^\prime}D^\beta_{x^\prime}D^j_\mu 
   k(x^\prime,\xi^\prime;\mu)\|_{\scrL(\cz,H^{s,\delta}(\rz_+))}
   \preceq \spk{\xi^\prime,\mu}^{d+\frac{1}{2}-|\alpha|-j}
\end{equation}
for every $s,\delta\in\rz$ and all orders of derivatives. In other terms, 
 $$k\in \mathop{\mbox{\large$\cap$}}_{s,\delta\in\rz}
   S^{d+\frac{1}{2}}_{1,0}(\rz^{n-1}\times\rpbar;\cz,H^{s,\delta}(\rz_+)),$$
where $H^{s,\delta}(\rz_+)$ is equipped with the standard group-action from Example $\ref{ex:standard-group-action}$, while on $\cz$ we consider the trivial group-action $\kappa\equiv1$. 
\end{theorem}
\begin{proof}
	Let us first start out from a Poisson symbol-kernel. 
	Then, by Theorem \ref{thm:poisson2nd} and chain-rule, 
	$D^\alpha_{\xi^\prime}D^\beta_{x^\prime}D^j_\mu k$ is a linear-combination of terms of the form 
	 $$p_{i,j^\prime,\alpha^\prime}(\xi^\prime;\mu)\spk{\xi^\prime,\mu}^{1/2}
	 (D^\beta_{x^\prime}(x_nD_{x_n})^iD^{\alpha^{\prime\prime}}_{\xi^\prime}
	   D^{j^{\prime\prime}}_\mu \mathfrak{k})(x^\prime,\xi^\prime,\spk{\xi^\prime,\mu}x_n;\mu)$$ 
    with $i+j^\prime+j^{\prime\prime}+|\alpha^{\prime}|+|\alpha^{\prime\prime}|=|\alpha|+j$ and 
	symbols $p_{i,j^\prime,\alpha^\prime}\in S^{-i-j^\prime-|\alpha^\prime|}_{1,0}(\rz^{n-1}\times\rpbar)$. Therefore the left-hand side 
	of \eqref{eq:abc} can be estimated by a linear combination of terms 
	 $$\spk{\xi^\prime,\mu}^{-i-j^\prime-|\alpha^\prime|}\|D^\beta_{x^\prime}D^i_{x_n}
	 D^{\alpha^{\prime\prime}}_{\xi^\prime}D^{j^{\prime\prime}}_\mu\mathfrak{k}\|_{H^{s,\delta}(\rz_+)}.$$
	Then \eqref{eq:abc} follows from the fact that 
	$\mathfrak{k}\in \mathfrak{S}^{d+\frac{1}{2}}_{1,0}(\rz^{n-1}\times\rpbar;\scrS(\rz_+))$. 
	
	Now consider $k$ as in \eqref{eq:xyz} satisfying \eqref{eq:abc}, identified with 
	a function $k(x^\prime,x_n,\xi^\prime;\mu)$. Define $\mathfrak{k}=\kappa^{-1}(\xi^\prime,\mu)k$, i.e., 
	 $$\mathfrak{k}(x^\prime,x_n,\xi^\prime;\mu)=\spk{\xi^\prime,\mu}^{-1/2}
	   k(x^\prime,\xi^\prime,\spk{\xi^\prime,\mu}^{-1}x_n;\mu).$$
	Again using chain rule, $D^\alpha_{\xi^\prime}D^\beta_{x^\prime}D^j_\mu \mathfrak{k}$ is a linear combination of terms 
	 $$p_{i,j^\prime,\alpha^\prime}(\xi^\prime;\mu)(x_nD_{x_n})^i\kappa^{-1}(\xi^\prime,\mu)    
	   D^\beta_{x^\prime}D^i_{x_n}D^{\alpha^{\prime\prime}}_{\xi^\prime}
       D^{j^{\prime\prime}}_\mu {k}$$ 
    with indices and $p_{i,j^\prime,\alpha^\prime}$ as before. This implies 
    $\mathfrak{k}\in \mathfrak{S}^{d+\frac{1}{2}}_{1,0}(\rz^{n-1}\times\rpbar;\scrS(\rz_+))$, as desired.    
\end{proof}

Similarly we can treat weakly parameter-dependent Poisson operators$:$

\begin{theorem}\label{thm:charact-weakly}
	A function $\wt{k}$ from \eqref{eq:xyz} corresponds to a weakly parameter-dependent Poisson symbol-kernel of order $d+\frac{1}{2}$ and regularity $\nu$ if, and only if, 
	$$\|\kappa^{-1}(\xi^\prime)D^\alpha_{\xi^\prime}D^j_\mu \wt{k}(\xi^\prime;\mu)\|_{
	\scrL(\cz,\calK^{s,0}(\rz_+)^\rho)}
	\preceq \spk{\xi^\prime}^{\nu-|\alpha|}\spk{\xi^\prime,\nu}^{d+\frac{1}{2}-\nu-j}$$ 
	and 
	$$\|\kappa^{-1}(\xi^\prime,\mu)D^\alpha_{\xi^\prime}D^j_\mu \wt{k}(\xi^\prime;\mu)\|_{\scrL(\cz,H^{s}(\rz_+))}
	\preceq \spk{\xi^\prime}^{\nu-|\alpha|}\spk{\xi^\prime,\nu}^{d+\frac{1}{2}-\nu-j}$$ 
	for every $s$, $\rho$ and all orders of derivatives, where $\kappa$ is the standard group-action from Example $\ref{ex:standard-group-action}$. 
\end{theorem}

\subsection{Homogeneous Poisson symbol-kernels}

A Poisson symbol-kernel which is strictly homogeneous of degree $d-\frac{1}{2}$ and has regularity 
$\nu$ (in the sense of Grubb) is a kernel $k(x^\prime,x_n,\xi^\prime,\mu)$, defined for $\xi^\prime\not=0$ only, satisfying  
\begin{equation}\label{eq:poisson-hom}
 k(x^\prime,x_n/\lambda,\lambda\xi^\prime;\lambda\mu)=\lambda^{d+\frac{1}{2}} k(x^\prime,x_n,\xi^\prime;\mu)\qquad \forall\,\lambda>0, 
\end{equation} 
and  
$$\|D^{\beta}_{x^\prime}D^{\alpha}_{\xi^\prime}D^j_\mu x_n^\ell D_{x_n}^{\ell^\prime}k\|_{L_2(\rz_+)}
\preceq \Big[\Big(\frac{|\xi^\prime|}{|\xi^\prime,\mu|}\Big)^{\nu-[\ell-\ell^\prime]_+-|\alpha|}+1\Big]
|\xi^\prime,\mu|^{d-\ell+\ell^\prime-|\alpha|-j}$$
for every order of derivatives and every $\ell$. 

We can repeat the above discussion and introduce strongly parameter-dependent homogeneous kernels 
by requiring the estimates 
$$\|D^{\beta}_{x^\prime}D^{\alpha}_{\xi^\prime}D^j_\mu x_n^\ell D_{x_n}^{\ell^\prime}k\|_{L_2(\rz_+)}
\preceq |\xi^\prime,\mu|^{d-\ell+\ell^\prime-|\alpha|-j}$$
$($and $k$ being defined for $(\xi^\prime,\mu)\not=0)$ and weakly parameter-dependent homogeneous kernels by requiring  
$$\|D^{\beta}_{x^\prime}D^{\alpha}_{\xi^\prime}D^j_\mu x_n^\ell D_{x_n}^{\ell^\prime}k\|_{L_2(\rz_+)}
\preceq \Big(\frac{|\xi^\prime|}{|\xi^\prime,\mu|}\Big)^{\nu-[\ell-\ell^\prime]_+-|\alpha|}
|\xi^\prime,\mu|^{d-\ell+\ell^\prime-|\alpha|-j}.$$

Then Theorem \ref{thm:charact-strongly} and Theorem \ref{thm:charact-weakly} have a corresponding version in the homogeneous setting $($using an identification with an operator-valued function as described in the beginning of Section \ref{sec:opval-poisson}$)$: Strongly homogeneous kernels are characterized by the estimates 
\begin{equation*}\label{eq:hom-abc}
\|\kappa^{-1}_{|\xi^\prime,\mu|} D^\alpha_{\xi^\prime}D^\beta_{x^\prime}D^j_\mu 
k\|_{\scrL(\cz,H^{s,\delta}(\rz_+))}
\preceq |\xi^\prime,\mu|^{d-|\alpha|-j}; 
\end{equation*}
weakly homogeneous kernels are characterized by simultaneous validity of the estimates 
 $$\|\kappa^{-1}_{|\xi^\prime|}D^\alpha_{\xi^\prime}D^j_\mu \wt{k}\|_{
   \scrL(\cz,\calK^{s,0}(\rz_+)^\rho)}
   \preceq |\xi^\prime|^{\nu-|\alpha|}|\xi^\prime,\nu|^{d-\nu-j}$$ 
and 
 $$\|\kappa^{-1}_{|\xi^\prime,\mu|}D^\alpha_{\xi^\prime}D^j_\mu   
   \wt{k}\|_{\scrL(\cz,H^{s}(\rz_+))}
   \preceq |\xi^\prime|^{\nu-|\alpha|}|\xi^\prime,\nu|^{d-\nu-j}.$$ 
Note that the homogeneity relation \eqref{eq:poisson-hom} then can be rephrased as 
\begin{equation}\label{eq:poisson-hom2}
 k(x^\prime,\lambda\xi^\prime;\lambda\mu)
 =\lambda^{d}\kappa_\lambda  k(x^\prime,\xi^\prime;\mu)\qquad \forall\,\lambda>0. 
\end{equation}  

\section{The algebra of singular Green symbols}\label{sec:green}

\subsection{Other classes of operator-valued symbols}\label{sec:green01}
Let $E_0$ and $E_1$ be Hilbert spaces equipped with group actions $\kappa_0$ and $\kappa_1$, respectively. 

Recall from Definition \ref{def:weakly} that $\wt{S}^{d,\nu}_{1,0}(\rz^{n}\times\rpbar;E_0,E_1)_{\bfkappa,\bfkappa}$ 
denotes the space of all smooth $\scrL(E_0,E_1)$-valued functions $a$ satisfying the 
estimates 
\begin{equation}\label{eq:bfkappa}
\|\kappa^{-1}_1(\xi,\mu)\big\{D^\alpha_{\xi}D^\beta_{x}D^j_\mu
a(x,\xi;\mu)\big\}\kappa_0(\xi,\mu)\|_{\scrL(E_0,E_1)}
\preceq \spk{\xi}^{\nu-|\alpha|}\spk{\xi,\mu}^{d-\nu-j}.
\end{equation} 

However, in Theorem \ref{thm:charact-weakly} we have seen that naturally arise operator-valued symbols involving the group-action $\kappa(\xi^\prime)$ rather than $\kappa(\xi^\prime,\mu)$. In order to capture this feature on the level of operator-valued symbols we shall introduce 
three further symbol spaces, by modifying the above defining estimate \eqref{eq:bfkappa}$:$
\begin{itemize}
	\item[i$)$] Substituting $\kappa_0(\xi,\mu)$ by $\kappa_0(\xi)$ yields the space 
	 $$\wt{S}^{d,\nu}_{1,0}(\rz^{n}\times\rpbar;E_0,E_1)_{\kappa,\bfkappa}.$$
	\item[ii$)$] Substituting $\kappa_1^{-1}(\xi,\mu)$ by $\kappa_1^{-1}(\xi)$ yields the 
	space 
	$$\wt{S}^{d,\nu}_{1,0}(\rz^{n}\times\rpbar;E_0,E_1)_{\bfkappa,\kappa}.$$
	\item[iii$)$] Substituting both $\kappa_0(\xi,\mu)$ and $\kappa_1^{-1}(\xi,\mu)$ by 
	$\kappa_0(\xi)$ and $\kappa_1^{-1}(\xi)$, respectively, yields the space 
	$$\wt{S}^{d,\nu}_{1,0}(\rz^{n}\times\rpbar;E_0,E_1)_{\kappa,\kappa}.$$
\end{itemize}
Note the difference between $\bfkappa$ $($written in "boldface"$)$, indicating the use of $\kappa_j(\xi,\mu)$ for $j=0$ or $j=1$, and $\kappa$, indicating the use of $\kappa_j(\xi)$ 
for $j=0$ or $j=1$. 

Modifying in the analogous way Definition \ref{def:kappa-hom} of the space of $\kappa$-homogeneous 
functions $\wt{S}^{(d,\nu)}(\rz^{n}\times\rpbar;E_0,E_1)_{\bfkappa,\bfkappa}$, 
leads to the three spaces 
\begin{align*}
\wt{S}^{(d,\nu)}&(\rz^{n}\times\rpbar;E_0,E_1)_{\kappa,\bfkappa},\\ 
\wt{S}^{(d,\nu)}&(\rz^{n}\times\rpbar;E_0,E_1)_{\bfkappa,\kappa},\\ 
\wt{S}^{(d,\nu)}&(\rz^{n}\times\rpbar;E_0,E_1)_{\kappa,\kappa}.
\end{align*}
Correspondingly, there are three classes of classical symbols, denoted by 
$$\wt{S}^{d,\nu}(\rz^{n}\times\rpbar;E_0,E_1)_{\kappa,\bfkappa},\quad 
\wt{S}^{d,\nu}(\rz^{n}\times\rpbar;E_0,E_1)_{\bfkappa,\kappa},\quad 
\wt{S}^{d,\nu}(\rz^{n}\times\rpbar;E_0,E_1)_{\kappa,\kappa}.$$ 
In any of the three cases we have the notion of homogeneous principal symbol, defined by 
$$a^{(d,\nu)}(x,\xi;\mu)=\lim_{\lambda\to+\infty}
\lambda^{-d}\kappa_{1,\lambda}^{-1}a(x,\lambda\xi;\lambda\mu)\kappa_{0,\lambda}.$$

The calculus for such symbols is pretty much the same as described in 
Section \ref{sec:oscillatory} but, 
additionally, one has to keep track of the indices $\kappa$ and $\bfkappa$, respectively. 

Concerning the Leibniz product of symbols $($respectively the composition of pseudodifferential operators$)$, one has to pay attention that the $\kappa$-subscripts "fit together"; for example,  
if $a_0\in \wt{S}^{d_0,\nu_0}(\rz^{n}\times\rpbar;E_0,E_1)_{\kappa,\bfkappa}$ and 
$a_1\in \wt{S}^{d_1,\nu_1}(\rz^{n}\times\rpbar;E_1,E_2)_{\bfkappa,\kappa}$, then 
$a_1\# a_0\in \wt{S}^{d_0+d_1,\nu_0+\nu_1}(\rz^{n}\times\rpbar;E_0,E_2)_{\kappa,\kappa}$. 

Taking the formal adjoint leads to an interchange of the two $\kappa$-indices; for example, 
if $a\in \wt{S}^{d,\nu}(\rz^{n}\times\rpbar;E_0,E_1)_{\kappa,\bfkappa}$ then
$a^{(*)}\in \wt{S}^{d,\nu}(\rz^{n}\times\rpbar;E_1,E_0)_{\bfkappa,\kappa}$. 

Also asymptotic summations can be performed in any of these classes. 

\subsection{Singular Green symbols of type $0$}\label{sec:green02}

Below we shall work with various Hilbert spaces of the form $\calE(\rz_+)\oplus\cz^N$ with 
$\calE(\rz_+)\subset\scrD^\prime(\rz_+)$. On such a space we shall consider the 
standard group-action, extended trivially to $\cz^N$, i.e.,  
$$\kappa_\lambda(u\oplus z)=(\kappa_\lambda u)\oplus z,\qquad 
  u\in\calE(\rz_+),\; z\in\cz^N.$$ 

First we focus on symbols of type $($also called "class" in the literature$)$ $0$. 

Strongly parameter-dependent symbols coincide with the parameter-dependent Green-symbols 
of Schulze's version of the parameter-dependent Boutet de Monvel algebra, cf., for instance, \cite[Section 2]{ScSc1}. The following definition uses the formalism of \cite[Section 3.2]{Schr01}. 

\begin{definition}\label{def:b}
	$\calB^{d;0}_G(\rz^{n-1}\times\rpbar;(L,M),(L^\prime,M^\prime))$ denotes the space 
	$$\mathop{\mbox{\Large$\cap$}}_{s,s^\prime,\delta,\delta^\prime\in\rz} 
	S^d(\rz^{n-1}\times\rpbar;H^{s,\delta}_0(\rpbar,\cz^{L})\oplus\cz^{M},
	H^{s^\prime,\delta^\prime}(\rz_+,\cz^{L^\prime})\oplus\cz^{M^\prime}).$$ 
    For convenience, we shall use the short-hand notation 
    $\calB^{d;0}_G((L,M),(L^\prime,M^\prime))$. 
\end{definition}

The space 
  $$\calB^{(d);0}_G(\rz^{n-1}\times\rpbar;(L,M),(L^\prime,M^\prime))
    =\calB^{(d);0}_G((L,M),(L^\prime,M^\prime))$$ 
of strongly $\kappa$-homogeneous symbols is definied analogously 
by taking the intersection over the corresponding spaces of homogeneous symbols. 


The weakly parameter-dependent class is defined as follows$:$

\begin{definition}\label{def:btilde}
	Denote by $\wt\calB^{d,\nu;0}_G(\rz^{n-1}\times\rpbar;(L,M),(L^\prime,M^\prime))$ 
	the space of all symbols that, for every choice of $s,s^\prime,\delta,\delta^\prime$, belong 
	to each of the following four spaces$:$
	\begin{align*}
	\wt{S}^{d,\nu}&(\rz^{n-1}\times\rpbar;H^{s}_0(\rpbar,\cz^{L})\oplus\cz^{M},
	H^{s^\prime}(\rz_+,\cz^{L^\prime})\oplus\cz^{M^\prime})_{\bfkappa,\bfkappa},\\ 
	\wt{S}^{d,\nu}&(\rz^{n-1}\times\rpbar;H^{s}_0(\rpbar,\cz^{L})\oplus\cz^{M},
	\calK^{s^\prime,0}(\rz_+,\cz^{L^\prime})^{\delta^\prime}\oplus
	\cz^{M^\prime})_{\bfkappa,\kappa},\\ 
	\wt{S}^{d,\nu}&(\rz^{n-1}\times\rpbar;\calK^{s,0}(\rz_+,\cz^{L})^\delta\oplus\cz^{M},
	H^{s^\prime}(\rz_+,\cz^{L^\prime})\oplus\cz^{M^\prime})_{\kappa,\bfkappa},\\ 
	\wt{S}^{d,\nu}&(\rz^{n-1}\times\rpbar;\calK^{s,0}(\rz_+,\cz^{L})^\delta\oplus\cz^{M},
	\calK^{s^\prime,0}(\rz_+,\cz^{L^\prime})^{\delta^\prime}\oplus\cz^{M^\prime})_{\kappa,\kappa}.  
	\end{align*}
    For convenience, we shall use the short-hand notation 
    $\wt\calB^{d,\nu;0}_G((L,M),(L^\prime,M^\prime))$. 
\end{definition}

Again, in the obvious way, we can define the corresponding space 
 $$\wt\calB^{(d,\nu);0}_G(\rz^{n-1}\times\rpbar;(L,M),(L^\prime,M^\prime))
   =\wt\calB^{(d,\nu);0}_G(\rz^{n-1}\times\rpbar;(L,M),(L^\prime,M^\prime))$$  
of weakly $\kappa$-homogeneous symbols. 

\begin{remark}
	Recall that the $L^2$-inner product allows to identify $H^{-s,-\delta}_0(\rpbar)$ with the 
	dual space of $H^{s,\delta}(\rz_+)$ and $\calK^{-s,0}(\rz_+)^{-\delta}$ with the dual space of 
	$\calK^{s,0}(\rz_+)^{\delta}$ for every choice of $s,\delta\in\rz$. 
	Hence it is clear from the definition of the symbol spaces, that taking pointwise the adjoint induces maps 
	\begin{align*} 
	\calB^{d;0}_G((L,M),(L^\prime,M^\prime)) 
	 & \to \calB^{d;0}_G((L^\prime,M^\prime),(L,M))\\
	\wt\calB^{d,\nu;0}_G((L,M),(L^\prime,M^\prime)) 
	 & \to \wt\calB^{d,\nu;0}_G((L^\prime,M^\prime),(L,M)). 
	\end{align*}
\end{remark}

Next we shall introduce the class of singular Green symbols of finite regularity$:$

\begin{definition}
	Denote by 
	$$\calB^{d,\nu;0}_G(\rz^{n-1}\times\rpbar;(L,M),(L^\prime,M^\prime))
	=\calB^{d,\nu;0}_G((L,M),(L^\prime,M^\prime))$$
	the $($non-direct$)$ sum 
	$$\calB^{d;0}_G((L,M),(L^\prime,M^\prime))+\wt\calB^{d,\nu;0}_G((L,M),(L^\prime,M^\prime)).$$ 
\end{definition}

Given an operator $A:E_0\oplus E_1\to F_0\oplus F_1$ acting in between direct sums of ceratin spaces, 
we may write $A$ in block-matrix form, i.e.,  
 $$A=\begin{pmatrix}A_{11}&A_{12}\\A_{21}& A_{22}\end{pmatrix}: 
          \begin{matrix}E_0\\ \oplus\\ E_1\end{matrix}
          \lra
          \begin{matrix}F_0\\ \oplus\\ F_1\end{matrix}.$$
In this sense, we may represent any symbol  
$\bfg\in \calB^{d,\nu;0}_G((L,M),(L^\prime,M^\prime))$ in block-matrix form, i.e., 
$$\bfg(x^\prime,\xi^\prime;\mu)
=\begin{pmatrix}
g(x^\prime,\xi^\prime;\mu) & k(x^\prime,\xi^\prime;\mu)\\
t(x^\prime,\xi^\prime;\mu) & q(x^\prime,\xi^\prime;\mu)
\end{pmatrix};
$$
then $g$ is called a singular Green symbol, $k$ is a Poisson symbol, and $t$ is a trace symbol. 
Clearly, $q$ is a usual $($matrix-valued$)$ pseudodifferential symbol.

The following observation is crucial for the calculus.  

\begin{proposition}\label{prop:incl}
	We have the inclusion 
	$$\calB^{d;0}_G((L,M),(L^\prime,M^\prime))\subset
	\wt\calB^{d,0;0}_G((L,M),(L^\prime,M^\prime))$$
and the analogous inclusion for the spaces of $\kappa$-homogeneous symbols.  
\end{proposition}
\begin{proof}
	For simplicity of notation let $L=L^\prime=1$ and $M=M^\prime=0$; the general case is treated in the same way and involves only more lengthy notation. Let 
	$$\bfg\in \mathop{\mbox{\Large$\cap$}}_{s,s^\prime,\delta,\delta^\prime\in\rz} 
	S^d(H^{s,\delta}_0(\rpbar),H^{s^\prime,\delta^\prime}(\rz_+)).$$ 
	For $k\in\nz_0$, norms on $\calK^{k,0}(\rz_+)^k$ and 
	$H^{k,k}(\rz_+)$ are given by 
	$$\|u\|=\sum\limits_{\substack{0\le \ell^\prime\le k\\ \ell^\prime\le\ell\le k}}
	\|t^\ell D^{\ell^\prime}_t u\|_{L^2(\rz_+)},\qquad  
	\|v\|=\sum\limits_{0\le \ell,\ell^\prime\le k}
	\|t^\ell D^{\ell^\prime}_t v\|_{L^2(\rz_+)},$$
	respectively; in particular, $H^{k,k}(\rz_+)\hookrightarrow \calK^{k,0}(\rz_+)^k$. 
	Using these norms, it is easy to see that 
	$$\|\kappa_\lambda\|^{(k)}:=\|\kappa_\lambda\|_{\scrL(H^{k,k}(\rz_+),\calK^{k,0}(\rz_+)^k)}
	\le \max(1,\lambda^{-k}).$$
	By duality it then follows that $\calK^{-k,0}(\rz_+)^{-k}\hookrightarrow H^{-k,-k}_0(\rpbar)$ and 
	$$\|\kappa_\lambda\|^{(-k)}:=\|\kappa_\lambda\|_{
		\scrL(\calK^{-k,0}(\rz_+)^{-k},H^{-k,-k}_0(\rpbar))}
	\le\max(1,\lambda^k).$$
	By assumption, 
	$\bfg\in \wt{S}^{d,0}(H^{s,\delta}_0(\rpbar),H^{k,k}(\rz_+))_{\bfkappa,\bfkappa}$. 
	Therefore 
	\begin{align*}
	\|\kappa^{-1}(\xi^\prime)&\{D^{\alpha}_{\xi^\prime}D^{\beta}_{x^\prime}D^j_\mu \bfg(x^\prime,\xi^\prime;\mu)\}\kappa(\xi^\prime,\mu)
	\|_{\scrL(H^{s,\delta}_0(\rpbar),\calK^{k,0}(\rz^+)^k)} \\
	\preceq & \|\kappa_{\spk{\xi^\prime,\mu}/\spk{\xi^\prime}}\|^{(k)}\cdot\\
	&\cdot\|\kappa^{-1}(\xi^\prime,\mu)\{D^{\alpha}_{\xi^\prime}D^{\beta}_{x^\prime}D^j_\mu \bfg(x^\prime,\xi^\prime;\mu)\}\kappa(\xi^\prime,\mu)
	\|_{\scrL(H^{s,\delta}_0(\rpbar),H^{k,k}(\rz^+))}\\
	\preceq & \spk{\xi^\prime}^{-|\alpha|}\spk{\xi^\prime,\mu}^{d-j},
	\end{align*}
	since $\spk{\xi^\prime,\mu}/\spk{\xi^\prime}\ge 1$. 
	Choosing $k\ge \max(s^\prime,\delta^\prime)$ we conclude that 
	$\bfg$ belongs to $\wt{S}^{d,0}_{1,0}(H^{s}_0(\rpbar),
	\calK^{s^\prime,0}(\rz_+)^{\delta^\prime})_{\bfkappa,\kappa}$. 
	
	Using $\spk{\xi^\prime}/\spk{\xi^\prime,\mu}\le 1$,  
	\begin{align*}
	\|\kappa^{-1}(\xi^\prime,\mu)&\{D^{\alpha}_{\xi^\prime}D^{\beta}_{x^\prime}D^j_\mu \bfg(x^\prime,\xi^\prime;\mu)\}\kappa(\xi^\prime)
	\|_{\scrL(\calK^{-k,0}(\rz^+)^{-k},H^{s^\prime}(\rz_+))} \\
	\preceq& 
	\|\kappa^{-1}(\xi^\prime,\mu)\{D^{\alpha}_{\xi^\prime}D^{\beta}_{x^\prime}D^j_\mu \bfg(x^\prime,\xi^\prime;\mu)\}\kappa(\xi^\prime,\mu)
	\|_{\scrL(H^{-k,-k}_0(\rz^+),H^{s^\prime}(\rpbar))}\cdot\\
	&\cdot\|\kappa_{\spk{\xi^\prime}/\spk{\xi^\prime,\mu}}\|^{(-k)}\\
	\preceq& \spk{\xi^\prime}^{-|\alpha|}\spk{\xi^\prime,\mu}^{d-j}. 
	\end{align*}
	Choosing $k$ large enough, it follows that 
	$\bfg\in\wt{S}^{d,0}_{1,0}(
	\calK^{s,0}(\rz_+)^{\delta},H^{s^\prime}(\rz_+))_{\kappa,\bfkappa}$. 
	
	Analogously one verifies that 
	$\bfg\in\wt{S}^{d,0}_{1,0}
	\calK^{s,0}(\rz_+)^{\delta},\calK^{s^\prime,0}(\rz_+)^{\delta^\prime})_{\kappa,\kappa}$. 
	
	Homogeneous components and corresponding expansions are treated similarly and yield the result for classical symbols. 
\end{proof}

\begin{corollary}
Due to Proposition $\ref{prop:incl}$,  
 $$\calB^{d,\nu;0}_G((L,M),(L^\prime,M^\prime))=
   \wt\calB^{d,\nu;0}_G((L,M),(L^\prime,M^\prime)),\qquad\nu\le0.$$
\end{corollary}

Given a function $u$ on $\rz_+$,  denote by $e_+u$ its extension by $0$ to $\rz$. It is known that $e_+$ extends 
to a mapping on the Sobolev spaces $H^s(^rz_+)$ provided $s>-\frac{1}{2}$. More precisely, $e_+$ induces maps 
\begin{align*}
 e_+&:H^s(\rz_+)\to H^s_0(\rpbar),\qquad -\frac{1}{2}<s<\frac{1}{2},\\
 e_+&:H^s(\rz_+)\to H^{\frac{1}{2}-\eps}_0(\rpbar),\qquad s\ge\frac{1}{2},
\end{align*}
where $\eps>0$ can be chosen arbitrarily small. By $r_+$ we shall denote the operator of restriction 
$\scrD^\prime(\rz)\to\scrD^\prime(\rz_+)$. Obviously, the same is valid in a $\cz^L$-valued setting. 

An operator or an operator-valued symbol $a$ defined on the intersection over all $s$ of all spaces $H^s_0(\rpbar,\cz^L)$ 
canonically induces an operator or operator-valued symbol defined on $H^s(\rz_+,\cz^L)$, $s>-1/2$, defined as $ae_+$. 
In particular, this applies to the singular Green symbols, i.e., 
$\calB^{d,\nu;0}_G((L,M),(L^\prime,M^\prime))$ can be considered as a subspace of 
\begin{equation}\label{eq:123}
 {S}^{d,\nu}(H^{s}(\rz_+,\cz^{L})\oplus\cz^{M},
   H^{s^\prime}(\rz_+,\cz^{L^\prime})\oplus\cz^{M^\prime}),
   \qquad s>-\frac{1}{2},\;s^\prime\in\rz,
\end{equation}
in the sense of Definition $\ref{def:sdnu}$. Actually, this will be the standard way of viewing singular Green symbols. 

\subsection{Leibniz product and formal adjoint symbol}\label{sec:green04}

Given a symbol $\bfg\in\calB^{d,\nu;0}_G((L,M),(L^\prime,M^\prime))$ and some fixed $\wh\mu\ge 0$, 
 $$\wh{\bfg}:=\bfg(\cdot,\cdot;\wh\mu)\in 
	S^d(\rz^{n-1}_{\xi^\prime};H^{s,\delta}_0(\rpbar,\cz^{L})\oplus\cz^{M},
           H^{s^\prime,\delta^\prime}(\rz_+,\cz^{L^\prime})\oplus\cz^{M^\prime}).$$  
for every choice of $s,s^\prime,\delta,\delta^\prime\in\rz$. In other words, $\wh\bfg$ is a singular 
Green symbol of type $0$ in Boutet de Monvel's algebra without parameter. This can be seen by writing 
 $$\wh{\bfg}(x^\prime,\xi^\prime)=\sum_{i,j=1}^2 
   \begin{pmatrix}\omega_i(\xi^\prime)&0\\0&1\end{pmatrix}
   \wh{\bfg}(x^\prime,\xi^\prime)
   \begin{pmatrix}\omega_j(\xi^\prime)&0\\0&1\end{pmatrix},$$
where $\omega\in \scrC^\infty_{\mathrm{comp}}([0,+\infty))$ is a cut-off function with $\omega\equiv1$ near 
the origin and $\omega_1(\xi^\prime)$ denotes the operator of multiplication by  $\omega(\cdot\spk{\xi^\prime})$, 
while $\omega_2(\xi^\prime)$ is the operator of multiplication by $(1-\omega)(\cdot\spk{\xi^\prime})$. In particular, 
it follows that 
 $$\op(\bfg)(\mu):
   \begin{matrix}\scrS(\rz^n_+,\cz^L)\\ \oplus \\ \scrS(\rz^{n-1},\cz^M)\end{matrix}
   \lra 
   \begin{matrix}\scrS(\rz^n_+,\cz^{L^\prime}) \\ \oplus \\  \scrS(\rz^{n-1},\cz^{M^\prime})\end{matrix}  
   \qquad\forall\;\mu\ge0. 
 $$ 

Using Theorems \ref{thm:comp02} and \ref{thm:adj02}, the observations stated in the end of 
Section \ref{sec:green01}, and Proposition \ref{prop:incl}, the behaviour of singular Green symbols 
under composition $($Leibniz product$)$ and formal adjoint is as follows: 

\begin{theorem}\label{thm:comp_type0}
If $\bfg_j\in \calB^{d_j,\nu_j;0}_G((L_j,M_j),(L_{j+1},M_{j+1}))$, $j=0,1$, then 
 $$\bfg_1\#\bfg_0\in  \calB^{d,\nu;0}_G((L_0,M_0),(L_{2},M_{2}))$$
with 
 $$d=d_0+d_1,\quad \nu=\min\{\nu_0,\nu_1,\nu_0+\nu_1\}.$$
Moreover, for every $N\in\nz$, 
 $$\bfg_1\#\bfg_0-\sum_{|\alpha|=0}^{N-1}\frac{1}{\alpha!}
     \partial^\alpha_{\xi^\prime}\bfg_1 D^\alpha_{x^\prime}\bfg_0
     \in  \calB^{d-N,\nu-N;0}_G((L_0,M_0),(L_{2},M_{2})).$$
\end{theorem}

If $\bfg_j\in \wt\calB^{d_j,\nu_j;0}_G((L_j,M_j),(L_{j+1},M_{j+1}))$ for at least one of the values 
$j=0$ or $j=1$, then $\bfg_1\#\bfg_0\in \wt\calB^{d,\nu;0}_G((L_0,M_0),(L_{2},M_{2}))$. In this sense, 
the weakly parameter-dependent Green symbols form  a two-sided ideal in the class of 
Green symbols. Moreover, the composition of two strongly parameter-dependent Green symbols is 
also strongly parameter-dependent.

\begin{theorem}
If $\bfg\in \calB^{d,\nu;0}_G((L,M),(L^\prime,M^\prime))$ then 
 $$\bfg^{(*)}\in  \calB^{d,\nu;0}_G((L^\prime,M^\prime),(L,M));$$
for every $N\in\nz$, 
 $$\bfg^{(*)}-\sum_{|\alpha|=0}^{N-1}\frac{1}{\alpha!}
     \partial^\alpha_{\xi^\prime}D^\alpha_{x^\prime}\bfg^*
     \in  \calB^{d-N,\nu-N;0}_G((L^\prime,M^\prime),(L,M)).$$
\end{theorem}

\section{Ellipticity and parametrix in the class ${1+\calB^{0,\nu;0}_G}$}\label{sec:param}

With every $\bfg=\bfg_0+\wt\bfg\in \calB^{d,\nu;0}_G((L,M),(L^\prime,M^\prime))$ 
we associate the homogeneous principal symbol  
\begin{equation}\label{eq:def_sigma}
 \sigma^{d,\nu}(\bfg):=\bfg_0^{(d)}+\wt{\bfg}^{(d,\nu)}.
\end{equation}
Note that, in particular, 
$$\sigma^{d,\nu}(\bfg) \in 
{S}^{(d,\nu)}(H^{s}_0(\rpbar,\cz^{L})\oplus\cz^{M},
H^{s^\prime}(\rz_+,\cz^{L^\prime})\oplus\cz^{M^\prime}),
\qquad s,s^\prime\in\rz,$$
and  
$$\sigma^{d,\nu}(\bfg) \in 
{S}^{(d,\nu)}(H^{s}(\rpbar,\cz^{L})\oplus\cz^{M},
H^{s^\prime}(\rz_+,\cz^{L^\prime})\oplus\cz^{M^\prime}),
\qquad s>-\frac{1}{2},s^\prime\in\rz,$$
where the spaces of $\kappa$- homogeneous symbols on the respective right-hand sides are 
understood in the sense of  Definition \ref{def:sdnu}. 

Combining the proof of Proposition \ref{prop:principal-map} with the proof of  
Proposition \ref{prop:incl} one finds$:$

\begin{theorem}
	The principal symbol induces a surjective map 
	$$\sigma^{(d,\nu)}:
	\calB^{d,\nu;0}_G((L,M),(L^\prime,M^\prime))
	\lra 
	\calB^{(d,\nu);0}_G((L,M),(L^\prime,M^\prime))$$
	with kernel equal to $\calB^{d-1,\nu-1;0}_G((L,M),(L^\prime,M^\prime))$. 
\end{theorem}

We shall now discuss ellipticity and parametrix construction for operators of the form identity plus zero order singular 
Green operator. Again, it is  necessary to require positive regularity $\nu>0$. 

It is convenient to introduce the space  
\begin{equation}\label{eq:ulc}
 B^{d,\nu;0}_G (\rz^{n-1}\times\rpbar;L,L^{\prime})=B^{(0,\nu);0}_G (L,L^\prime):=
     \calB^{d,\nu;0}_G((L,0),(L^{\prime},0)).
\end{equation}
Or, in other words, $B^{d,\nu;0}_G (L,L^{\prime})$ is the space of the entries $g$ in the block matrices 
$\bfg=\begin{pmatrix}g&k\\t&q\end{pmatrix}\in \calB^{d,\nu;0}_G ((L,M),(L^{\prime},M^{\prime}))$ with arbitrary 
$M,M^{\prime}$. 

\begin{lemma}\label{lem:sandwich}
Let $\nu>0$ and $a\in  {S}^{(0,\nu)}(H^{s_0}(\rz_+,\cz^{L_1}), H^{s_0}(\rz_+,\cz^{L_2}))$ 
for some fixed $s_0>-1/2$ be a $\kappa$-homogeneous symbol. Further let 
$g\in B^{(0,\nu);0}_G (L_0,L_1)$ and $h\in B^{(0,\nu);0}_G (L_2,L_3)$. 
Then $hag\in B^{(0,\nu);0}_G (L_0,L_3)$.
\end{lemma}
\begin{proof}
For convenience of notation we shall assume $L_0=L_1=L_2=L_3=1$; the general case is proved in the same way. 
Write $a=a_0+\wt{a}$ with 
 $$a_0\in S^{(0)}(H^{s_0}(\rz_+),H^{s_0}(\rz_+)),\qquad 
    \wt{a}\in \wt{S}^{(0,\nu)}(H^{s_0}(\rz_+),H^{s_0}(\rz_+))_{\bfkappa,\bfkappa}.$$ 
Similarly decompose $g$ and $h$ with 
$g_0,h_0\in B^{(0);0}_G(1,1)$ and $\wt{g},\wt{h}\in \wt{B}^{(0,\nu);0}_G(1,1)$. 
Recall that also $g_0,h_0\in \wt{B}^{(0,0);0}_G(1,1)$. Now write 
 $$hag=h_0a_0g_0+h_0\wt{a}g_0+h_0a\wt{g}+\wt{h}ag.$$
Since $h_0\in S^{(0)}(H^{s_0}(\rz_+),H^{s^\prime,\delta^\prime}(\rz_+))$ for every $s^\prime$ and $\delta^\prime$, 
it follows that $h_0a_0g_0\in B^{(0);0}_G(1,1)$. 

All the other terms in the above expression for $hag$ belong to $\wt{B}^{(0,\nu);0}_G(1,1)$. 
For example let us consider 
 $$h_0a\wt{g}=h_0a_0\wt{g}+h_0\wt{a}\wt{g}.$$
Since $h_0\in \wt{B}^{(0,0);0}_G(1,1)$, $\wt{g}\in\wt{B}^{(0,\nu);0}_G(1,1)$, and 
$a_0,\wt{a}\in\wt{S}^{(0,0)}(H^{s_0}(\rz_+),H^{s_0}(\rz_+))_{\bfkappa,\bfkappa}$ it is easy to see 
that both $h_0a_0\wt{g}$ and $h_0\wt{a}\wt{g}$ have all the mapping properties listed in Definition 
\ref{def:btilde}. The reasoning for all remaining terms in the expansion of $hag$ is analogous. 
\end{proof}

\begin{theorem}\label{thm:ellipticity-ulc}
Let $\nu>0$ and $g\in B^{0,\nu;0}_G (L,L)$. Suppose that, for some $s_0>-1/2$,  
 $$1+\sigma^{(0,\nu)}(g)(x^\prime,\xi^\prime;\mu):H^{s_0}(\rz_+,\cz^L)\lra H^{s_0}(\rz_+,\cz^L)$$
is invertible for all $x^\prime$ and $(\xi^\prime,\mu)\not=0$ with 
 $$\|(1+\sigma^{(0,\nu)}(g)(x^\prime,\xi^\prime;\mu))^{-1}\|_{\scrL(H^{s_0}(\rz_+,\cz^L))}\preceq 1
	\qquad \forall\;x^\prime\quad \forall\;|\xi^\prime,\mu|=1.$$
Then there exists an $h\in B^{0,\nu;0}_G (L,L)$ such that 
 $$(1+g)\#(1+h)-1,\,(1+h)\#(1+g)-1\;\in\;B^{-\infty,\nu-\infty;0}_G (L,L).$$
In particular, 
 $$(1+\sigma^{(0,\nu)}(g)(x^\prime,\xi^\prime;\mu))^{-1}
   =1+\sigma^{(0,\nu)}(h)(x^\prime,\xi^\prime;\mu).$$
\end{theorem}
\begin{proof}
For simplicity of notation write $\sigma(g)=\sigma^{(0,\nu)}(g)$. 
Proposition \ref{prop:inverse_principal} applied to 
 $$a:=1+\sigma(g)\in S^{(0,\nu)}(H^{s_0}(\rz_+,\cz^L),H^{s_0}(\rz_+,\cz^L))$$ 
yields the existence of a $b\in S^{(0,\nu)}(H^{s_0}(\rz_+,\cz^L),H^{s_0}(\rz_+,\cz^L))$ with $(1+\sigma(g))^{-1}=1+b$. 
Since   
 $$1+b=1-\sigma(g)+\sigma(g)(1+\sigma(g))^{-1}\sigma(g)=1-\sigma(g)+\sigma(g)^2+\sigma(g)b\sigma(g),$$
it follows from Lemma \ref{lem:sandwich} that $b\in B^{(0,\nu);0}_G(L,L)$. 

Choose any symbol $h^\prime\in B^{0,\nu;0}_G(L,L)$ with $\sigma^{(0,\nu)}(h^\prime)=b$. Then 
  $$(1+g)\#(1+h^\prime)=1-r^\prime, \qquad r^\prime\in B^{-1,\nu-1;0}_G(L,L),$$
since $r^\prime$ has vanishing principal symbol. Choose any $r\in  B^{-1,\nu-1;0}_G(L,L)$ with 
$r\sim \sum_{j\ge 1}(r^\prime)^{\#j}$. By the standard von Neumann argument, 
 $$(1+h^\prime)\#(1+r)=1+h^\prime+r+h^\prime\#r=:1+h_R$$ 
is a right-parametrix of $1+g$. 
In the analogous way we construct a left-parametrix $1+h_L$. 
Then $1+h_L$ and $1+h_R$ coincide up to a regularizing error. Hence the claim of the theorem follows by 
choosing $h=h_R$ or $h=h_L$. 
\end{proof}

The situation of general block-matrices can be reduced to the previous situation:

\begin{theorem}\label{thm:ellipticity}
	Let $\nu>0$ and $\bfg\in \calB^{0,\nu;0}_G((L,M),(L,M))$. 
	Suppose that   
	$$\mathbf{1}+\sigma^{(0,\nu)}(\bfg)(x^\prime,\xi^\prime;\mu):
	  \begin{matrix} L^2(\rz_+,\cz^L)\\ \oplus \\ \cz^M\end{matrix} 
	  \lra 
	  \begin{matrix} L^2(\rz_+,\cz^L)\\ \oplus \\ \cz^{M}\end{matrix}$$
	is an isomorphism for all $x^\prime$ and $(\xi^\prime,\mu)\not=0$ with 
	$$\|(\mathbf{1}+\sigma^{(0,\nu)}(\bfg)(x^\prime,\xi^\prime;\mu))^{-1}\|\preceq 1
	\qquad \forall\;x^\prime\quad \forall\;|\xi^\prime,\mu|=1.$$
	Then there exists an $\bfh\in \calB^{0,\nu;0}_G((L,M),(L,M))$ such that 
	both $	(\mathbf{1}+\bfg)\#(\mathbf{1}+\bfh)-\mathbf{1}$ and 
	$(\mathbf{1}+\bfh)\#(\mathbf{1}+\bfg)-\mathbf{1}$ 
	belong to $\calB^{-\infty,\nu-\infty;0}_G((L,M),(L,M))$.
	In particular, 
	$$(\mathbf{1}+\sigma^{(0,\nu)}(\bfg)(x^\prime,\xi^\prime;\mu))^{-1}
	=\mathbf{1}+\sigma^{(0,\nu)}(\bfh)(x^\prime,\xi^\prime;\mu).$$
\end{theorem}
\begin{proof}
Let us set $\bfa:=\mathbf{1}+\sigma^{(0,\nu)}(\bfg)$. By assumption, $\bfa$ is pointwise 
invertible and there exist positive constants $c\le C$ such that 
 $$c\|u\|^2\le \|\bfa(x^\prime,\xi^\prime;\mu)u\|^2\le C\|u\|^2$$ 
uniformly in $u\in L^2(\rz_+,\cz^L)\oplus\cz^M$ and $x^\prime$ and $|\xi^\prime,\mu|=1$. 

Obviously, $\bfa^{-1}=[\bfa^*\bfa]^{-1}\bfa^*$. We write 
 $$\bfa^*\bfa =\mathbf{1}+\wt\bfg=C(\mathbf{1}-\bfs),\qquad s=1-C^{-1}\bfa^*\bfa;$$
here 
 $$\wt\bfg=\begin{pmatrix}\wt{g}&\wt{k}\\ \wt{t}&\wt{q}\end{pmatrix}\in \calB^{(0,\nu);0}_G((L,M),(L,M)).$$
Note that $\bfs(x^\prime,\xi^\prime;\mu)=1-C^{-1}(\bfa^*\bfa)(x^\prime,\xi^\prime;\mu)\le\delta$ with 
$\delta:=\frac{C-c}{C}<1$. Then  
\begin{align*} 
 0 &\le (\bfs_{11} u,u)_{L^2(\rz_+,\cz^L)} \\
   &=\left([1-C^{-1}\bfa^*\bfa]\begin{pmatrix}u\\0\end{pmatrix},
   \begin{pmatrix}u\\0\end{pmatrix}\right)_{L^2(\rz_+,\cz^L)\oplus\cz^M}
   \le \delta(u,u)_{L^2(\rz_+,\cz^L)}
\end{align*} 
shows that 
 $$\|\bfs_{11}(x^\prime,\xi^\prime;\mu)\|_{\scrL(L^2(\rz_+,\cz^L))}\le \delta
   \qquad\forall\;x^\prime\quad\forall\;|\xi^\prime,\mu|=1.$$
Therefore $1+\wt{g}$ is pointwise invertible with 
 $$\|(1+\wt{g})^{-1}\|=C^{-1}\|(1-\bfs_{11})^{-1}\|\le C^{-1}\frac{1}{1-\delta}=\frac{1}{c}
      \qquad\forall\;x^\prime\quad\forall\;|\xi^\prime,\mu|=1.$$
By Theorem \ref{thm:ellipticity-ulc}, there exists an 
$\wt{h}\in B^{(0,\nu);0}_G(L,L)$ with 
$(1+\wt{g})^{-1}=1+\wt{h}$.  

Similarly one sees that 
 $$\|\bfs_{22}(x^\prime,\xi^\prime;\mu)\|_{\scrL(\cz^M)}\le\frac{1}{c}  
   \qquad\forall\;x^\prime\quad\forall\;|\xi^\prime,\mu|=1.$$
It follows that 
 $$(1+\wt{q})^{-1}=1+\wt{p},\qquad \wt{p}\in S^{(0,\nu)}(\rz^{n-1}\times\rpbar;\cz^M,\cz^M).$$ 
Therefore $\bfa^*\bfa=\begin{pmatrix}1+\wt{g}&\wt{k}\\ \wt{t}&1+\wt{q}\end{pmatrix}$ is invertible with 
 $$(\bfa^*\bfa)^{-1}
   =\begin{pmatrix}1&0\\-(1+\wt{p})\wt{t}&1+\wt{p}\end{pmatrix} 
    \begin{pmatrix}1+\wt{h}&0\\0&1\end{pmatrix}
    \begin{pmatrix}1&-\wt{k}(1+\wt{p})\\0&1\end{pmatrix}=\mathbf{1}+\bfc,
   $$
where 
 $$\bfc=\begin{pmatrix}
         \wt{h}&-(1+\wt{h})\wt{k}(1+\wt{p})\\
         -(1+\wt{p})\wt{t}(1+\wt{h})&\wt{p}+(1+\wt{p})\wt{t}(1+\wt{h})\wt{k}(1+\wt{p})
        \end{pmatrix}
        \in \calB^{(0,\nu);0}_G((L,M),(L,M)).$$
We conclude the existence of an $\bfh^\prime \in \calB^{(0,\nu);0}_G((L,M^\prime),(L,M))$ such that 
 $$(\mathbf{1}+\sigma^{(0,\nu)}(\bfg))^{-1}=\mathbf{1}+\sigma^{(0,\nu)}(\bfh^\prime).$$ 
Now one proceeds as in the proof of Theorem \ref{thm:ellipticity-ulc} to construct the parametrix.  
\end{proof}

\subsection{Invertibility for large parameter}\label{sec:inv_large_param}

In the previous section we have constructed the parametrix, i.e., an inverse modulo regularizing 
operators. In the present section we shall show that elliptic elements are invertible for sufficiently 
large values of $\mu$ and that one can find a parametrix that furnishes the exact inverse for 
these values of $\mu$. 

\begin{lemma}\label{lem:invlarge01}
Let $E$ be a Hilbert space equipped with trivial group-action $\kappa\equiv1$. 
Let $a\in \wt{S}^{-\infty,\nu-\infty}(\rpbar\times\rz^n;E,E)$ with some $\nu>0$. 
Then there exists a symbol $b$ of the same form and a constant $R\ge0$ such that 
 $$(1-a)\#(1-b)=(1-a)\#(1-b)=1  \qquad \forall\;\mu\ge R.$$
\end{lemma}
\begin{proof}
Note that we have the identification 
 $$\wt{S}^{-\infty,\nu-\infty}(\rpbar\times\rz^n;E,E)
   =S^{-\nu}_{1,0}(\rpbar;S^{-\infty}(\rz^n;E,E)).$$
As shown in Corollary \ref{cor:specinv02} of the appendix, 
 $$\calA:=\{1+\op(r)\mid r\in S^{-\infty}(\rz^n;E,E)\}\subset\scrL(L^2(\rz^n,E))$$
is a spectrally invariant Fr\'echet algebra $($with multiplication $\#$, i.e., 
the Leibniz product$)$. In particular, $\calA$ has an open group of invertible 
elements and inversion is a continuous operation in $\calA$ due to a classical result of 
Waelbroeck \cite{Wa3}. 

 It follows that $(1-a(\mu))^{-\#}$ $($i.e., the inverse with respect to the product $\#)$ 
 exists for $\mu\ge R$ for some suitable $R$ and that 
 $b_0(\mu):=\chi(\mu)(1-a(\mu))^{-\#}\in S^0_{1,0}(\rpbar;S^{-\infty}(\rz^n;E,E))$ 
 if $\chi$ is a $0$-excision function vanishing on $[0,R]$. 
 But then the identity $(1-T)^{-1}=1+T(1-T)^{-1}$ shows that, for large $\mu$, 
 the inverse of $1-a(\mu)$ coincides with $1-b(\mu)$, where 
$b(\mu)=-a(\mu)\#b_0(\mu)\in S^{-\nu}_{1,0}(\rpbar;S^{-\infty}(\rz^n;E,E))$. 
\end{proof}

\begin{proposition}\label{prop:inv02}
Let $\bfg\in \calB^{-\infty,\nu-\infty;0}_G((L,M),(L,M))$ with $\nu>0$.
Then there exists a $\bfh\in \calB^{-\infty,\nu-\infty;0}_G((L,M),(L,M))$ 
and a constant $R\ge0$ such that 
\begin{equation}\label{eq:inv01}
 (1-\bfg)\#(1-\bfh)=(1-\bfh)\#(1-\bfg)=1  \qquad \forall\;\mu\ge R.
\end{equation}
\end{proposition} 
\begin{proof}
Note that 
 $$\bfg\in \wt{S}^{-\infty,\nu-\infty}(\rpbar\times\rz^n;
   L^2(\rz,\cz^L)\oplus\cz^M,L^2(\rz,\cz^L)\oplus\cz^M))$$
and recall that $\kappa$ is a group of unitary operators on $L^2(\rz_+,\cz^L)\oplus\cz^M$, 
hence can be replaced here by the trivial group-action. According to Lemma \ref{lem:invlarge01} 
we find an $\bfh_0$ from the same space which inverts $1-\bfg$ in the sense of \eqref{eq:inv01}. 
Thus, for large $\mu$, 
 $$1-\bfh_0=(1-\bfg)^{-\#}=1+\bfg-\bfg\#(1-\bfg)^{-\#}\#\bfg= 1+\bfg-\bfg\#\bfg+\bfg\#\bfh_0\#\bfg.$$
Using Definition \ref{def:btilde} and the calculus of operator-valued symbols, it is clear that 
 $$\bfh:=-\bfg+\bfg\#\bfg-\bfg\#\bfh_0\#\bfg$$ 
belongs to $B^{-\infty,\nu-\infty;0}_G((L,M),(L,M))$. The result follows, 
since $1-\bfh=1-\bfh_0$ for large $\mu$. 
\end{proof}

\begin{theorem}
Using the notation of Theorem $\ref{thm:ellipticity}$, there exists a parametrix 
$\boldsymbol{1}+\bfh_0$ that coincides, for sufficiently large values of $\mu$, with the inverse 
$($with respect to the Leibniz product$)$ of $\boldsymbol{1}+\bfg$. 
\end{theorem}
\begin{proof}
Take $\bfh$ from Theorem $\ref{thm:ellipticity}$. Then 
$(\mathbf{1}+\bfg)\#(\mathbf{1}+\bfh)=1-\bfr_0$ with $\bfr\in \calB^{-\infty,\nu-\infty;0}_G((L,M),(L,M))$ 
and similarly $(\mathbf{1}+\bfh)\#(\mathbf{1}+\bfg)=1+\bfr_1$. 
In view of Proposition \ref{prop:inv02} it suffices to take $\bfh_0$ such that 
$\mathbf{1}+\bfh_0=(\mathbf{1}+\bfh)\#(1+\bfr_0)^{-\#}$. 
\end{proof}

\section{Singular Green symbols of non trivial type}\label{sec:green03}
 
Recall that the maps 
 $$u\mapsto \frac{d^ju}{dt^j}(0):\scrS(\rz_+,\cz^L)\lra\cz^L,\qquad j\in\nz_0,$$
extend by continuity to maps 
 $$\gamma_j:H^{s,\delta}(\rz_+,\cz^L)\lra\cz^L,\qquad s>j+\frac{1}{2},\;\delta\in\rz,$$
(we shall use the same notation, independent on the value of $L)$. 
We can consider $\gamma_j$ as an operator-valued symbol, constant in $(x^\prime,\xi^\prime;\mu)$; 
the following result is well-known, cf. \cite[Example 1.5]{Schr01}, for instance. 

\begin{proposition}\label{prop:trace}
For every $j\in\nz_0$, 
$\gamma_j$ is $\kappa$-homogeneous of degree $j+\frac{1}{2}$. In particular, 
 $$\gamma_j \in S^{j+\frac{1}{2}}(H^{s,\delta}(\rz_+,\cz^L),\cz^L),\qquad s>j+\frac{1}{2},\;\delta\in\rz.$$  
\end{proposition}  

It is convenient to introduce the notation 
$\displaystyle \boldsymbol{\gamma}_j
:=\begin{pmatrix}0&0\\\gamma_j&0\end{pmatrix}$.

\begin{definition}\label{def:type}
Let $r\in\nz$ be a positive integer. Then 
 $$\calB^{d,\nu;r}_G(\rz^{n-1}\times\rpbar;(L,M),(L^\prime,M^\prime))
     =\calB^{d,\nu;r}_G((L,M),(L^\prime,M^\prime))$$
denotes the space of all symbols of the form 
 $$\bfg=\bfg_0+\sum_{j=0}^{r-1}\begin{pmatrix}0& k_j\\ 0& q_j\end{pmatrix}
     \boldsymbol{\gamma}_j
     =\bfg_0+\sum_{j=0}^{r-1} 
     \begin{pmatrix}k_j\gamma_j&0\\ q_j\gamma_j&0\end{pmatrix},$$
where $\bfg_0\in \calB^{d,\nu;0}_G((L,M),(L^\prime,M^\prime))$ and  
   $\begin{pmatrix}0& k_j\\ 0& q_j\end{pmatrix}\in 
     \calB^{d-j-\frac{1}{2},\nu;0}_G((L,L),(L^\prime,M^\prime))$. 
The number $r$ is called the type of the symbol $\bfg$.  
\end{definition}

For a symbol $\bfg$ as in the previous definition, the principal symbol is defined as 
 $$\sigma^{(d,\nu)}(\bfg)=
     \sigma^{(d,\nu)}(\bfg_0)+\sum_{j=0}^{r-1}
     \begin{pmatrix} 0&\sigma^{(d,\nu)}(k_j)\\ 0&\sigma^{(d,\nu)}(q_j)\end{pmatrix}
     \boldsymbol{\gamma}_j,$$

There is another, equivalent definition of symbols of type $r>0$: 
If $\partial_+$ is the operator of differentiation defined on $\scrD^\prime(\rz_+,\cz^L)$, 
it can be checked easily that $\partial_+$ induces constant operator-valued symbols 
\begin{align*}
\boldsymbol{\partial_+}
:=\begin{pmatrix}\partial_+&0\\0&1\end{pmatrix}\in  S^1(H^{s,\delta}(\rz_+,\cz^L)\oplus\cz^M,
H^{s-1,\delta}(\rz_+,\cz^L)\oplus\cz^{M})
\end{align*}
for every $s,\delta\in\rz$. Then $\calB^{d,\nu;r}_G((L,M),(L^\prime,M^\prime))$ consists of all symbols of the form 
 $$\bfg=\bfg_0+\sum_{j=1}^{r}\bfg_j\boldsymbol{\partial}_+^j,\qquad 
   \bfg_j\in \calB^{d-j,\nu;0}_G((L,M),(L^\prime,M^\prime)).$$
Using integration by parts, symbols from Definition \ref{def:type} can be represented in this 
form, cf. \cite[Example 3.10]{Schr01}, for instance. 
For the reverse direction one needs to use a characterization of 
singular Green, Poisson, and trace symbols in terms of parameter-dependent integral kernels 
$($as we have seen in Theorem \ref{thm:charact-weakly} for Poisson symbols). 
However, since we shall not make use of this alternative representation, we shall not proof this equivalence here. 

\begin{lemma}\label{lem:dg}
Let $\bfg_0\in \calB^{d_0,\nu_0;0}_G((L_0,M_0),(L_{1},M_{1}))$ and 
$\bfg_1\in \calB^{d_1,\nu_1;0}_G((L_1,L_1),(L_{2},M_{2}))$. Then 
 $$(\bfg_1\boldsymbol{\gamma}_k)\#\bfg_0\in \calB^{d_0+d_1+k+\frac{1}{2},\nu;0}_G((L_0,M_0),(L_{2},M_{2})),
     \qquad \nu=\min\{\nu_0,\nu_1,\nu_0+\nu_1\}.$$
\end{lemma}
\begin{proof}
Let us first remark that the Leibniz product of $\boldsymbol{\gamma}_k$ with another operator-valued symbol 
$($acting in appropriate spaces$)$ coincides with the pointwise product of both, 
since $\boldsymbol{\gamma}_k$ is constant. Therefore, 
$(\bfg_1\boldsymbol{\gamma}_k)\#\bfg_0=\bfg_1\#\boldsymbol{\gamma}_k\#\bfg_0=
\bfg_1\#(\boldsymbol{\gamma}_k\bfg_0)$. 

It is enough to study the two cases where either both $\bfg_0$ and $\bfg_1$ are strongly parameter-dependent 
or both are weakly parameter-dependent; the mixed cases then follow from Proposition \ref{prop:incl}. 

For simplicity of notation let us now assume that $M_0=M_1=M_2=0$. 
Consider first the case of strongly 
parameter-dependent symbols. By Proposition \ref{prop:trace},
 $$\boldsymbol{\gamma}_k\bfg_0\in S^{d_0+k+\frac{1}{2}}(\rz^{n-1}\times\rpbar; 
     H^{s,\delta}_0(\rz_+,\cz^{L_0}),H^{s^\prime,\delta^\prime}(\rz_+,\cz^{L_1})\oplus \cz^{L_1})$$
for every $s,s^\prime,\delta,\delta^\prime\in\rz$. 
Recall that $\bfg_1$ induces a symbol 
 $$\bfg_1\in S^{d_1}(\rz^{n-1}\times\rpbar; 
     H^{s,\delta}(\rz_+,\cz^{L_1})\oplus\cz^{L_1},H^{s^\prime,\delta^\prime}(\rz_+,\cz^{L_2}))$$
for every $s>-1/2$ and all $s^\prime,\delta,\delta^\prime\in\rz$. Then the result follows simply by using the 
Leibniz product of strongly parameter-dependent operator-valued symbols, cf. Section \ref{sec:oscillatory} 
and Example \ref{ex:Leibniz}. 

In case of weak parameter-dependence we have to check that the Leibniz-product belongs to all four spaces 
listed in Definition \ref{def:btilde}. For example, due to Proposition \ref{prop:trace},  
 $$\boldsymbol{\gamma}_k\bfg_0\in \wt{S}^{d_0+k+\frac{1}{2},\nu_0}( 
     \calK^{s,0}(\rz_+,\cz^{L_0})^\delta,
     H^{s^\prime,\delta^\prime}(\rz_+,\cz^{L_1})\oplus\cz^{L_1})_{\kappa,\bfkappa}.$$
Moreover, cf. the paragraph before \eqref{eq:123}, 
 $$\bfg_1\in \wt{S}^{d_1,\nu_1}( H^{s,\delta}(\rz_+,\cz^{L_1})\oplus\cz^{L_1},
     \calK^{s^\prime,0}(\rz_+,\cz^{L_2})^{\delta^\prime})_{\bfkappa,\kappa}$$
for every $s>-1/2$ and all $s^\prime,\delta,\delta^\prime\in\rz$. By using the 
Leibniz product of weakly parameter-dependent operator-valued symbols as described in the end of 
Section \ref{sec:green01} it follows that 
 $$\bfg_1\#\boldsymbol{\partial}_+^k\bfg_0\in \wt{S}^{d_0+d_1+k,\nu_0+\nu_1}
     (\calK^{s,0}(\rz_+,\cz^{L_0})^\delta,
     \calK^{s^\prime,0}(\rz_+,\cz^{L_2})^{\delta^\prime})_{\kappa,\kappa}.$$
The argument for the other three spaces listed in Definition \ref{def:btilde} is analogous. 
\end{proof}

It is now clear that Theorem \ref{thm:comp_type0} on the Leibniz product of Green symbols 
with type $0$ extends to general types as follows:

\begin{theorem}\label{thm:comp_typer}
If $\bfg_j\in \calB^{d_j,\nu_j;r_j}_G((L_j,M_j),(L_{j+1},M_{j+1}))$, $j=0,1$, then 
 $$\bfg_1\#\bfg_0\in  \calB^{d,\nu;r}_G((L_0,M_0),(L_{2},M_{2}))$$
with 
 $$d=d_0+d_1,\quad \nu=\min\{\nu_0,\nu_1,\nu_0+\nu_1\},\quad r=r_0.$$
Moreover, for every $N\in\nz$, 
 $$\bfg_1\#\bfg_0-\sum_{|\alpha|=0}^{N-1}\frac{1}{\alpha!}
     \partial^\alpha_{\xi^\prime}\bfg_1 D^\alpha_{x^\prime}\bfg_0
     \in  \calB^{d-N,\nu-N;r}_G((L_0,M_0),(L_{2},M_{2})).$$
\end{theorem}

\section{Action in Sobolev spaces}\label{sec:Sobolev}

In this section we discuss the mapping properties of parameter-dependent singular Green operators in 
Sobolev spaces. 

\subsection{The abstract framework}\label{sec:sobolev01}

Given a Hilbert space $E$ with group-action $\kappa$, let 
$\calW^s(\rz^n,E)$, $s\in\rz$, denote the space of all distributions 
$u\in\scrS^\prime(\rz^n,E)$ such 
that its Fourier transform $\wh{u}$ is a regular distribution and
  $$\|u\|_{\calW^s(\rz^n,E)}=\Big(\int\spk{\xi}^{2s}\|\kappa^{-1}(\xi)
  \hat{u}(\xi)\|_{E}^2\,d\xi\Big)^{1/2}<+\infty.$$
(note that the integral is defined, since together with $\wh{u}$ also 
$\kappa^{-1}(\cdot)\wh{u}$ is measurable due to the strong continuity the group-action$)$. 
These spaces are called abstract wedge Sobolev spaces and have been introduced by  
Schulze in \cite{Schu4}; cf. also \cite{Hirs1}. 

Now consider two Hilbert spaces $E_0$ and $E_1$ with respective group-actions $\kappa_0$ and $\kappa_1$. 
Let $S^d(\rz^n;E_0,E_1)$ be the space of all symbols $a(x,\xi)$  satisfying 
 $$\|\kappa_1^{-1}(\xi)\{D^\alpha_\xi D^\beta_x a(x,\xi)\}\kappa_0(\xi)\|_{\scrL(E_0,E_1)}
     \preceq\spk{\xi}^{d-|\alpha|}.$$
The following theorem follows easily from \cite[Theorem 3.14]{Seil99} together 
with the calculus for such operator-valued symbols. 

\begin{theorem}
Every $a\in S^d(\rz^n;E_0,E_1)$ induces continuous operators 
 $$\op(a):\calW^s(\rz^n,E_0)\lra \calW^{s-d}(\rz^n,E_1),\qquad s\in\rz.$$
\end{theorem}

Now let us introduce spaces with parameter-dependent norms. 

\begin{definition}
We denote by $\calW^{(s,t),\mu}(\rz^n,E)$ and $\calW^{(s,t),\bfmu}(\rz^n,E)$  
the Sobolev space $\calW^{s+t}(\rz^n,E)$ equipped with the norms 
 $$\|u\|_{\calW^{(s,t),\mu}(\rz^n,E)}=\Big(\int\spk{\xi,\mu}^{2s}\spk{\xi}^{2t}\|\kappa^{-1}(\xi)
     \hat{u}(\xi)\|_{E}^2\,d\xi\Big)^{1/2}$$
and
 $$\|u\|_{\calW^{(s,t),\bfmu}(\rz^n,E)}=\Big(\int\spk{\xi,\mu}^{2s}\spk{\xi}^{2t}\|\kappa^{-1}(\xi,\mu)
  \hat{u}(\xi)\|_{E}^2\,d\xi\Big)^{1/2},$$
respectively. 
\end{definition} 

Assume now that on a Hilbert space $E$ there exists a family of norms $\|\cdot\|_{\mu}$, 
parametrized by a parameter $\mu$, that are all equivalent 
to a fixed norm on $E$; denote by $E^\mu$ the space $E$ equipped with the norm  $\|\cdot\|_{\mu}$. Moreover, 
suppose we have a family $A(\mu)\in\scrL(E)$. We shall write $A(\mu)\in\scrL(E^\mu)$ if 
 $$\|A(\mu)\|_{\scrL(E^\mu)}=\sup_{e\in E}\frac{\|A(\mu)e\|_{\mu}}{\|e\|_\mu}\preceq 1.$$
This concepts extends in the obvious way to two different spaces $E_0$ and $E_1$ equipped with $\mu$-dependent 
norms. 

\begin{theorem}\label{thm:mapping}
Let $s,t\in\rz$. The following statements are valid$:$
\begin{itemize}
 \item[a$)$] $a\in \wt{S}^{d,\nu}_{1,0}(\rz^n\times\rpbar;E_0,E_1)_{\bfkappa,\bfkappa}$ implies 
  $$\op(a)(\mu)\in\scrL\big(\calW^{(s,t),\bfmu}(\rz^n,E_0),
    \calW^{(s+\nu-d,t-\nu),\bfmu}(\rz^n,E_1)\big).$$  
 \item[b$)$] $a\in \wt{S}^{d,\nu}_{1,0}(\rz^n\times\rpbar;E_0,E_1)_{\kappa,\bfkappa}$ implies 
  $$\op(a)(\mu)\in\scrL\big(\calW^{(s,t),\mu}(\rz^n,E_0),
     \calW^{(s+\nu-d,t-\nu),\bfmu}(\rz^n,E_1)\big).$$  
 \item[c$)$] $a\in \wt{S}^{d,\nu}_{1,0}(\rz^n\times\rpbar;E_0,E_1)_{\bfkappa,\kappa}$ implies 
  $$\op(a)(\mu)\in\scrL\big(\calW^{(s,t),\bfmu}(\rz^n,E_0),
     \calW^{(s+\nu-d,t-\nu),\mu}(\rz^n,E_1)\big).$$  
 \item[d$)$] $a\in \wt{S}^{d,\nu}_{1,0}(\rz^n\times\rpbar;E_0,E_1)_{\kappa,\kappa}$ implies 
  $$\op(a)(\mu)\in\scrL\big(\calW^{(s,t),\mu}(\rz^n,E_0),
    \calW^{(s+\nu-d,t-\nu),\mu}(\rz^n,E_1)\big).$$  
 \item[e$)$] $a\in {S}^{d}_{1,0}(\rz^n\times\rpbar;E_0,E_1)$ implies 
  $$\op(a)(\mu)\in\scrL\big(\calW^{(s,t),\bfmu}(\rz^n,E_0),\calW^{(s-d,t),\bfmu}(\rz^n,E_1)\big).$$
\end{itemize}
\end{theorem}
\begin{proof}
a)--d) Let $\lambda^{d,\nu}_j(\xi;\mu)=\spk{\xi}^\nu\spk{\xi,\mu}^{d-\nu}\mathrm{id}_{E_j}$ and 
$\wt{a}=\lambda_1^{s-d+t,t-\nu}\#a\#\lambda_0^{-s-t,-t}$.
Then $\wt{a}$ is of order and regularity $0$ and the statements are equivalent to the continuity of $\op(\wt{a})(\mu)$ between the corresponding parameter-dependent Sobolev spaces of order $(0,0)$. 
However, this can be shown following the proof of \cite[Theorem 3.14]{Seil99}, making only 
simple modifications that regard the use of the semi-group function 
$\kappa_j(\xi,\mu)$ in place of $\kappa_j(\xi)$. 

e) is verified similarly. 
\end{proof}

There are various estimates relating different parameter-dependent norms. For purposes below we mention that 
\begin{equation}\label{eq:param-estimate}
 \|u\|_{\calW^{(s,t),\bfmu}(\rz^n,E)}\le 
 \spk{\mu}^{r_+} \|u\|_{\calW^{(s-r,t+r),\bfmu}(\rz^n,E)},\qquad u\in \calW^{s+t}(\rz^n,E), 
\end{equation}
for every $r,s,t,\in\rz$. This simply follows from the fact that 
 $$\spk{\xi,\mu}^s\spk{\xi}^t=\spk{\xi,\mu}^{s-r}\spk{\xi,\mu}^r\spk{\xi}^t\le     
   \spk{\xi,\mu}^{s-r}\spk{\xi}^{t+r}\spk{\mu}^{r_+}.$$ 

\subsection{Example: Sobolev spaces on the half-space}\label{sec:sobolev02}

We shall show how the previously described concept of edge Sobolev spaces allows to recover 
certain $($anisotropic$)$ Sobolev spaces with parameter-dependent norm. 
Let us first look at the full Euclidean space. Define 
\begin{align*}
 H^{(s,t)}(\rz^{n})=\{u\in\scrS^\prime(\rz^n) \mid \spk{\xi}^s\spk{\xi^\prime}^t\wh{u}(\xi)\in L^2(\rz^n)\}
\end{align*}
with obvious definition of the norm. The parameter-dependent version, denoted by 
$H^{(s,t),\mu}(\rz^{n})$, carries the norm(s) 
 $$\|u\|_{H^{(s,t),\mu}(\rz^{n})}=\|\spk{\xi,\mu}^s\spk{\xi^\prime,\mu}^t\wh{u}(\xi)\|_{L^2(\rz^n)}.$$
 
\begin{proposition}
With equality of norms, $H^{(s,t)}(\rz^{n})=\calW^{s+t}(\rz^{n-1},H^s(\rz))$ and 
 $$H^{(s,t),\mu}(\rz^{n})=\calW^{(s+t,0),\bfmu}(\rz^{n-1},H^s(\rz)).$$
\end{proposition}
\begin{proof}
By definition, 
 $$\|u\|_{H^{(s,t),\mu}(\rz^{n})}^2
   =\iint\spk{\xi^\prime,\xi_n,\mu}^{2s}\spk{\xi^\prime,\mu}^{2t}|\wh{u}(\xi^\prime,\xi_n)|^2
    \,d\xi_n d\xi^\prime.$$
By the change of variables $\xi_n=\spk{\xi^\prime,\mu}\tau_n$ and noting that 
$\spk{\xi^\prime,\spk{\xi^\prime,\mu}\tau_n,\mu}=\spk{\xi^\prime,\mu}\spk{\tau_n}$ we find 
 $$\|u\|_{H^{(s,t),\mu}(\rz^{n})}^2
   =\int\spk{\xi^\prime,\mu}^{2(s+t)}\Big(\int\spk{\tau_n}^{2s}|\spk{\xi^\prime,\mu}^{1/2}
    \wh{u}(\xi^\prime,\spk{\xi^\prime,\mu}\tau_n)|^2\,d\xi_n\Big) d\xi^\prime.$$
Noting that Fourier transform and group-action satisfy $\kappa_\lambda\scrF=\scrF\kappa_\lambda^{-1}$, we find 
 $$\|u\|_{H^{(s,t),\mu}(\rz^{n})}^2
   =\int\spk{\xi^\prime,\mu}^{2(s+t)}\|\kappa^{-1}(\xi^\prime,\mu)\scrF_{x^\prime\to\xi^\prime}u(\xi^\prime)\|_{H^s(\rz)}^2\,d\xi^\prime.$$
This is exactly what has been claimed. 
\end{proof}

Similarly, for the half-space $\rz^n_+=\rz^{n-1}\times\rz_+$, we have 
\begin{align}
 H^{(s,t),\mu}_0(\overline{\rz}^{n}_+)
 &=\calW^{(s+t,0),\bfmu}(\rz^{n-1},H^s_0(\rpbar)),\\
 H^{(s,t),\mu}(\rz^{n}_+)
 &=\calW^{(s+t,0),\bfmu}(\rz^{n-1},H^s(\rz_+)),\label{eq:h-w1}
\end{align}
where the first space consists of all elements of $H^{(s,t),\mu}(\rz^{n})$ having support 
contained in $\overline{\rz}^{n}_+$, while the second denotes the space of restrictions to ${\rz}^{n}_+$, cf. the discussion in Section \ref{sec:app01.1}. 

\begin{remark}
Recall that $\cz$ is equipped with the trivial group action $\kappa\equiv 1$. It is then 
rather obvious that 
\begin{equation}\label{eq:h-w2}
 \calW^{(s,0),\bfmu}(\rz^{n},\cz)=\calW^{(s,0),\mu}(\rz^{n},\cz)=H^{(s,0),\mu}(\rz^n).
\end{equation}
\end{remark}

All constructions from above extend trivially to $\cz^N$-valued Sobolev spaces. 

\subsection{Mapping properties of singular Green operators}\label{sec:sobolev03}

As mentioned already in \eqref{eq:123},  
$\calB^{d,\nu;r}_G((L,M),(L^\prime,M^\prime))$ is embedded in 
\begin{align*}
 S^{d,\nu}(H^{s}(&\rz_+,\cz^{L})\oplus\cz^{M},
 H^{s^\prime}(\rz_+,\cz^{L^\prime})\oplus\cz^{M^\prime})\\
 =&\,S^{d}(H^{s}(\rz_+,\cz^{L})\oplus\cz^{M},
H^{s^\prime}(\rz_+,\cz^{L^\prime})\oplus\cz^{M^\prime})+\\
  &+\wt{S}^{d,\nu}(H^{s}(\rz_+,\cz^{L})\oplus\cz^{M},
H^{s^\prime}(\rz_+,\cz^{L^\prime})\oplus\cz^{M^\prime})_{\bfkappa,\bfkappa} 
\end{align*}
for every choice of of $s>r-\frac{1}{2}$ and $s^\prime\in\rz$. Representing 
$a\in\calB^{d,\nu;r}_G((L,M),(L^\prime,M^\prime))$ as $a=a_0+\wt{a}$ in this sense, we 
obtain immediately the mapping properties stated in the parts a$)$ and e$)$ of Theorem \ref{thm:mapping}. 
In particular, employing the identifications \eqref{eq:h-w1} and \eqref{eq:h-w2}, we find that 
 $$\op(a_0)(\mu)\in 
   \scrL\left(
    \begin{matrix}
    H^{(s,t),\mu}(\rz^n_+,\cz^L)\\ \oplus\\ H^{(s+t,0),\mu}(\rz^{n-1},\cz^M)
    \end{matrix},
    \begin{matrix}
    H^{(s-d,t),\mu}(\rz^n_+,\cz^{L^\prime})\\ \oplus\\ H^{(s+t-d,0),\mu}(\rz^{n-1},\cz^{M^\prime}),
    \end{matrix}
   \right)$$
and 
 $$\op(\wt{a})(\mu)\in 
   \scrL\left(
   \begin{matrix}
    H^{(s,t),\mu}(\rz^n_+,\cz^L)\\ \oplus\\ H^{(s+t,0),\mu}(\rz^{n-1},\cz^M)
   \end{matrix},
   \begin{matrix}
    \calW^{(s+t+\nu-d,-\nu),\bfmu}(\rz^{n-1},H^{s+t-d}(\rz_+,\cz^{L^\prime}))
    \\ \oplus\\ \calW^{(s+t+\nu-d,-\nu),\bfmu}(\rz^{n-1},\cz^{M^\prime})
   \end{matrix}
   \right).$$
Now, using the estimate \eqref{eq:param-estimate} $($with $s$ subsituted by $s+t+d-\nu$, $t$ substituted by $0$, and $r$ substituted by $-\nu)$, we conclude the following: 

\begin{theorem}
Let $a\in\calB^{d,\nu;r}_G((L,M),(L^\prime,M^\prime))$ 
and $s,t\in\rz$ with $s>r-\frac{1}{2}$. Then 
 $$\frac{\op(a)(\mu)}{1+\spk{\mu}^{-\nu}}\in 
\scrL\left(
\begin{matrix}
H^{(s,t),\mu}(\rz^n_+,\cz^L)\\ \oplus\\ H^{(s+t,0),\mu}(\rz^{n-1},\cz^M)
\end{matrix},
\begin{matrix}
H^{(s-d,t),\mu}(\rz^n_+,\cz^{L^\prime})\\ \oplus\\ H^{(s+t-d,0),\mu}(\rz^{n-1},\cz^{M^\prime}),
\end{matrix}
\right).$$
\end{theorem}

\section{Appendix: A result on spectral invariance}\label{sec:app02}

We shall proof a characterization of pseudodifferential operators with operator-valued symbols 
in terms of mapping properties of certain commutators. In the scalar-valued setting, this result 
is a simple special case of a well-known theorem of Beals \cite[Theorem 1.4]{Beal77}. We derive 
spectral invariance of two psudodifferential algebras. 

Let $\mathfrak{A}\subset\scrL(X,Y)$ be a set of bounded operators between Banach spaces 
$X$ and $Y$. By uniform boundedness principle, the following three statements are equivalent$:$
\begin{itemize}
 \item[i$)$] $\mathfrak{A}$ is a bounded set. 
 \item[ii$)$] $\mathfrak{A}_x:=\{Ax\mid A\in\mathfrak{A}\}$ is a bounded subset of $Y$ 
  for every $x\in X$. 
 \item[iii$)$] $y^\prime(\mathfrak{A}_x)$ is a bounded set for every $x\in X$ and 
  $y^\prime\in Y^\prime$. 
\end{itemize}
If $E$ is a Hilbert space we thus obtain the following:

\begin{lemma}
Let $\mathfrak{A}\subset\scrL(E)$. Then $\mathfrak{A}$ is bounded if and only if 
$\{(Ae,f)\mid A\in\mathfrak{A}\}$ are bounded sets for every $e,f\in E$. 
\end{lemma}

\begin{corollary}\label{cor:smooth}
Let $u:\rz^n\to\scrL(E)$ be a function. Then $u$ is smooth if and only if the functions  
$(u(\cdot)e,f)$ are smooth for every $e,f\in E$. 
\end{corollary}
\begin{proof}
We claim that if all $(u(\cdot)e,f)$ are $(N+1)$-times continuously differentiable, then $u$ is 
$N$-times continuously differentiable. By Induction it suffices to show this for $N=0$. 
Given $x_0\in\rz$, define 
 $$\mathfrak{A}=\left\{\frac{u(x_0+h)-u(x_0)}{|h|}\mid 0<|h|\le1\right\}.$$
By assumption on $u$, all sets $\{(Ae,f)\mid A\in\mathfrak{A}\}$ are bounded, hence $\mathfrak{A}$ 
is bounded. Clearly this implies continuity of $u$ in $x_0$. 
\end{proof}

For an operator $B:\scrS(\rz^n,E)\to \scrS^\prime (\rz^n,E)$ let 
 $$L_kB=iBx_k-ix_kB, \qquad M_kB=D_{x_k}B-BD_{x_k}\qquad (1\le k\le n),$$
where $x_k$ means the operator of multiplikation by the function $x\mapsto x_k$. Moreover, set
 $$B^{(\alpha)}_{(\beta)}=L_1^{\alpha_1}\cdots L_n^{\alpha_n}M_1^{\alpha_1}\cdots M_n^{\alpha_n}B,
 \qquad \alpha,\beta\in\nz_0.$$ 

Let us denote by $S^m_{\rho,\delta}(\rz^n;E,E)$ the operator-valued H\"ormander class of all smooth functions $b:\rz^n\times\rz^n\to \scrL(E)$ satisfying, for every order of derivatives,  
 $$\|D^\alpha_\xi D^\beta_x b(x,\xi)\|\preceq \spk{\xi}^{m-\rho|\alpha|+\delta|\beta|}.$$
 
\begin{theorem}\label{thm:Beals}
For a linear operator $B:\scrS(\rz^n,E)\to \scrS^\prime (\rz^n,E)$ the following are equivalent$:$
\begin{itemize}
	\item[$(1)$] There exists a symbol $b(x,\xi)\in S^0_{0,0}(\rz^n;E,E)$ with $B=\op(b)$. 
	\item[$(2)$] All iterated commutators $B^{(\alpha)}_{(\beta)}$ extend to operators in 
	 $\scrL(L^2(\rz^n,E))$. 	 
\end{itemize}
\end{theorem}
\begin{proof}
In case $E=\cz$ this is a special case of \cite[Theorem 1.4]{Beal77}. Let us summarize the proof 
given there$:$ Let $h$ be a smooth cut-off function with $h\equiv 1$ in a neighborhood of the origin. Define 
 $$B_\eps=\op(q_\eps)B\op(p_\eps),\qquad q_\eps(x)=h(\eps x),\;p_\eps(\xi)=h(\eps\xi).$$
Let 
 $$b_{0,\eps}(x,y,\xi)=e_{-\xi}(x)B(e_\xi g(\cdot-y)),$$
where $e_\xi(x)=e^{ix\xi}$ and $g\in\scrS(\rz^n)$ is even with Fourier transform supported in the 
unit-ball centered at the origin. One shows that $b_{0,\eps}$ is a smooth function with 
 $$|D^\beta_x D^\gamma_y D^\alpha_\xi b_{0,\eps}(x,y,\xi)|\le C_{\alpha\beta\gamma}$$ 
uniformly in $(x,y,\xi)$ and $0<\eps\le1$, with constants that are (respectively, 
can be chosen to be) finite linear combinations in the $\scrL(L^2(\rz^n))$-operator norms of 
the iterated commutators of $B$. Then 
 $$b_\eps(x,\xi)=\iint e^{iy\eta}b_{0,\eps}(x,x+y,\xi+\eta)\,dy\dbar\eta$$
$($oscillatory integral$)$ is a smooth function satisfying analogous estimates. Then 
$b=\lim_{\eps\to0}b_{\eps}$ exists in $S^0_{0,0}(\rz^n)$ and $B=\op(b)$. Again, the constants in the 
symbol estimates are as described before. 

Now let us turn to the general case of a Hilbert space $E$. 
For an operator $T:\scrS(\rz^n,E)\to \scrS^\prime (\rz^n,E)$ and $e,f\in E$ define 
$T_{e,f}:\scrS(\rz^n)\to \scrS^\prime (\rz^n)$ by 
 $$\spk{T_{e,f}\phi,\psi}=\spk{T(\phi\otimes e),\psi\otimes f},\qquad 
   \phi,\psi\in\scrS(\rz^n),$$
where $\spk{\cdot,\cdot}$ is the pairing of distributions and test functions and 
$(\phi\otimes e)(x)=\phi(x)e$. 
Then $(B_{e,f})^{(\alpha)}_{(\beta)}=(B^{(\alpha)}_{(\beta)})_{e,f}$ 
and for the $L^2$-extension holds 
 $$\|(B_{e,f})^{(\alpha)}_{(\beta)}\|_{\scrL(L^2(\rz^n))}\le 
   \|B^{(\alpha)}_{(\beta)}\|_{\scrL(L^2(\rz^n,E))}\|e\|\|f\|.$$
By the scalar-valued result, there exists a symbol $b_{e,f}\in S^0_{0,0}(\rz^n)$ with $B_{e,f}=\op(b_{e,f})$ and 
\begin{equation}\label{eq:bilin}
 |D^\alpha_\xi D^\beta_x b_{e,f}(x,\xi)|\le C_{\alpha\beta}\|e\|\|f\|.
\end{equation}
Now define $b(x,\xi):H\to H$ by 
 $$\spk{b(x,\xi)e,f}=b_{e,f}(x,\xi),\qquad e,f\in H.$$
Then $b\in S^0_{0,0}(\rz^n;E,E)$ due to \eqref{eq:bilin} and Corollary \ref{cor:smooth}. Moreover, 
\begin{align*}
 \spk{\op(b)(\phi\otimes e),\psi\otimes f}
 &=\int\Big(\int e^{-ix\xi} b(x,\xi)\wh{\phi}(\xi)e\,\dbar\xi,\psi(x)f\Big)_E\,dx\\
 &=\iint e^{-ix\xi} (b(x,\xi)e,f)_E\wh{\phi}(\xi)\psi(x)\,dx\dbar\xi\\
 &=\spk{\op(b_{e,f})\phi,\psi}=\spk{B(\phi\otimes e),\psi\otimes f}. 
\end{align*}
Since finite linear combinations of functions of the form $u\otimes e$ with $u\in\scrS(\rz^n)$ and $e\in E$ are a dense subspace of $\scrS(\rz^n,E)=\scrS(\rz^n)\,\wh{\otimes}_\pi E$, we conclude that 
$B=\op(b)$. 
\end{proof}

An algebra $\calA\subset\scrL(E)$, $E$ an arbitrary Hilbert space, is called spectrally invariant 
if an element $a\in\calA$ is invertible in $\calA$ if and only if it is invertible in $\scrL(E)$, i.e., 
 $$\calA^{-1}=\calA\cap \scrL(E)^{-1},$$
where $\mathcal{X}^{-1}$ denotes the group of invertible elements of the algebra $\mathcal{X}$. 

\begin{corollary}\label{cor:specinv01}
Let $\calA=\big\{\op(a)\mid a\in S^0_{0,0}(\rz^n;E,E)\big\}$. 
Then $\calA\subset\scrL(L^2(\rz^n,E))$ is spectrally invariant. 
\end{corollary}
\begin{proof}
Let $A=\op(a):L^2(\rz^n,E)\to L^2(\rz^n,E)$ be an isomorphism with inverse $A^{-1}$. Then 
 $$L_k A^{-1}=i A^{-1}x_k-ix_kA^{-1}=A^{-1}\big(ix_kA-iAx_k\big)A^{-1}=-A^{-1}(L_k A)A^{-1}$$
and, analogously, $M_k A^{-1}=-A^{-1}(M_k A)A^{-1}$. It follows that an iterated commutator of 
$A^{-1}$ is a finite linear combination of products, where each factor is either $A^{-1}$ 
or an iterated commutator of $A$. Hence every such iterated commutator extends to a bounded map 
in $L^2(\rz^n,E)$. By Theorem \ref{thm:Beals}, $A^{-1}\in\calA$. 
\end{proof}

\begin{corollary}\label{cor:specinv02}
Let $\calA=\big\{1+\op(a)\mid a\in S^{-\infty}(\rz^n;E,E)\}$. 
Then $\calA\subset\scrL(L^2(\rz^n,E))$ is spectrally invariant. 
\end{corollary}
\begin{proof}
Let $A=1+\op(a):L^2(\rz^n,E)\to L^2(\rz^n,E)$ be an isomorphism. By the preevious corollary, 
$A^{-1}=\op(b_0)$ for some $b_0\in S^0_{0,0}(\rz^n;E,E)$. But then  
 $$\op(b_0)=(1+\op(a))^{-1}=1-\op(a)+\op(a)(1+\op(a))^{-1}\op(a)$$
shows that $A^{-1}=1+\op(b)$ with $b=-a+a\#b_0\#a\in S^{-\infty}(\rz^n;E,E)$. 
\end{proof}

\section{Appendix: Function spaces for the half-line}\label{sec:app01}

We recall here the definitions and some properties of all spaces of functions or distributions 
we shall need throughout this paper.  

\subsection{Bessel potential spaces}\label{sec:app01.1}
We denote by $H^s(\rz)=H^s_2(\rz)$, $s\in\rz$, the standard $L^2$-Sobolev paces, consisting of 
those tempered distributions $u\in\scrS^\prime(\rz)$ whose Fourier transform is a measurable function with 
$$\|u\|_{H^s(\rz)}:=\|\spk{\cdot}^s\wh{u}\|_{L^2(\rz)}<+\infty.$$
The subspace of those distributions whose support is a subset of $\rpbar:=[0,+\infty)$ is denoted by $H^s_0(\rpbar)$, 
$$H^s_0(\rpbar)=\{u\in H^s(\rz)\mid \mathrm{supp}\,u\subseteq\rpbar\},\qquad s\in\rz.$$
Obviously, $H^s_0(\rpbar)$ is a closed subspace of $H^s(\rz)$. Similarly, one definess $H^s(\overline{\rz}_-)$ with 
$\overline{\rz}_-:=(-\infty,0]$. Moreover, 
$$H^s(\rz_+)=\{v\in\scrD^\prime(\rz_+)\mid \exists\,u\in H^s(\rz):\; u|_{\rz_+}=v\}$$
is obtained by the restriction of Sobolov distributions from $\rz$ to $\rz_+$; it carries the norm 
$$\|v\|_{H^s(\rz_+)}:=\inf\{\|u\|_{H^s(\rz)}\mid u\in H^s(\rz),\; u|_{\rz_+}=v\}.$$
Obviously, 
$$H^s(\rz_+)\cong H^s(\rz)/H^s(\overline{\rz}_-).$$

\subsection{Cone Sobolev spaces}\label{sec:app01.2}

The change of variables $s=e^{-t}$ induces induces an isomorphism 
$\theta:\scrD^\prime(\rz_+)\to\scrD^\prime(\rz)$ by 
$$\spk{\theta u,\varphi}=\spk{u,\varphi(-\ln s)/s},\qquad \varphi\in\scrD(\rz).$$
Then we define 
$$\scrH^{s,\frac{1}{2}}(\rz_+):=\theta^{-1}H^s(\rz), \qquad 
\|u\|_{\scrH^{s,\frac{1}{2}}(\rz_+)}=\|\theta u\|_{H^s(\rz)}.$$
Multiplication with powers $t^{\gamma-1/2}$ yields the spaces 
$$\scrH^{s,\gamma}(\rz_+):=t^{\gamma-\frac{1}{2}}\scrH^{s,\frac{1}{2}}(\rz_+),\qquad 
\|u\|_{\scrH^{s,\gamma}(\rz_+)}=\|t^{\frac{1}{2}-\gamma} u\|_{\scrH^{s,\frac{1}{2}}(\rz)}.$$
Note that for for $k\in\nz_0$ we have 
$$u\in \scrH^{s,\gamma}(\rz_+) \iff t^{\frac{1}{2}-\gamma}(t\partial_t)^j u\in L^2(\rz_+,dt/t)\quad\forall\;1\le j\le k.$$
Now let $\omega\in\scrC^\infty_{\mathrm{comp}}(\rpbar)$ be a cut-off function with $\omega\equiv 1$ near $t=0$. Then 
\begin{align*}
u\in\calK^{s,\gamma}(\rz_+):\iff \omega u\in \scrH^{s,\gamma}(\rz_+)\text{ and } (1-\omega)u\in H^s(\rz_+) 
\end{align*}
with norm 
$$\|u\|_{\calK^{s,\gamma}(\rz_+)}=\| \omega u\|_{\scrH^{s,\gamma}(\rz_+)}+ \|(1-\omega)u\|_{H^s(\rz_+)}$$
defines the Hilbert spaces 
$ \calK^{s,\gamma}(\rz_+)$, $s,\gamma\in\rz$.
Up to equivalence of norms, this construction is independent on the coice of $\omega$. Finally, we consider spaces with power-weight at infinity, 
$$H^{s,\delta}_{0}(\rpbar):=\spk{t}^{-\delta}H^{s}_{0}(\rpbar),\qquad 
H^{s,\delta}(\rz_+):=\spk{t}^{-\delta}H^{s}(\rz_+),$$
as well as 
$$\calK^{s,\gamma}(\rz_+)^{\delta}:=\spk{t}^{-\delta}\calK^{s,\gamma}(\rz_+)$$
with obvious definition of the corresponding norms. 

\subsection{Duality}\label{sec:app01.3}

The standard $L^2(\rz_+)$-inner product induces non-degenerate sesqui-linear pairings 
$$H^{s,\delta}(\rz_+)\times H^{-s,-\delta}_0(\rpbar)\lra\cz,\qquad 
\calK^{s,\gamma}(\rz_+)^{\delta}\times \calK^{-s,-\gamma}(\rz_+)^{-\delta}\to\cz$$
which allows for the following identification of dual spaces$:$
$$H^{s,\delta}(\rz_+)^\prime=H^{-s,-\delta}_0(\rpbar),\qquad 
(\calK^{s,\gamma}(\rz_+)^{\delta})^\prime=\calK^{-s,-\gamma}(\rz_+)^{-\delta}.$$

\subsection{Rapidly decreasing functions}\label{sec:app01.4}

Let $I=\rz$ or $I=\rz_+$ and $E$ a Hilbert space. We denote by $\scrS(I,E)$ the space of rapidly decreasing functions, 
i.e., of all smooth functions $u:I\to E$ with 
 $$\sup_{t\in I}\spk{t}^N \|D^k_t u(t)\|<+\infty\qquad \forall\;k,N\in\nz_0.$$  
 
\begin{small}

\end{small}

\end{document}